\documentclass[11pt]{amsart}
\usepackage{amssymb}
\usepackage{amsmath}
\usepackage{amsfonts}
\usepackage{graphicx}

\usepackage[total={17cm,22cm},top=2.5cm, left=2.3cm]{geometry}
\usepackage{hyperref}
    \usepackage{aeguill}
    \usepackage{type1cm}

\usepackage[shortlabels]{enumitem}

\newtheorem{thm}{Theorem}[section]

\newtheorem{lem}[thm]{Lemma}

\newtheorem{prop}[thm]{Proposition}
\newtheorem{rem}[thm]{Remark}
\newtheorem{ex}[thm]{Example}
\newtheorem{prob}[thm]{Problem}

\newtheorem{defn}[thm]{Definition}

\newtheorem{claim}[thm]{Claim}

\numberwithin{equation}{section}

\usepackage{amsxtra}
\usepackage{eucal}
\usepackage{mathrsfs}
\usepackage{longtable}
\usepackage{caption}

\usepackage{hyperref}
    \usepackage{aeguill}
    \usepackage{type1cm}

\newcommand{\N}{\mathbb{N}}
\newcommand{\Z}{\mathbb{Z}}

\newcommand{\R}{\mathbb{R}}

\usepackage[usenames,dvipsnames]{color}

\usepackage[dvipsnames]{xcolor}

\newcommand{\Bl}{\color{blue}}

\begin{document}

\title{Locally $C^{1,1}$ convex extensions of $1$-jets}

\author{Daniel Azagra}
\address{ICMAT (CSIC-UAM-UC3-UCM), Departamento de An{\'a}lisis Matem{\'a}tico y Matem\'atica Aplicada,
Facultad Ciencias Matem{\'a}ticas, Universidad Complutense, 28040, Madrid, Spain.  
}
\email{azagra@mat.ucm.es}
%
%
%

\date{May 15, 2019}

\keywords{convex function, convex body, locally $C^{1,1}$, Whitney extension theorems}


\begin{abstract}
Let $E$ be an arbitrary subset of $\mathbb{R}^n$, and $f:E\to\mathbb{R}$, $G:E\to\mathbb{R}^n$ be given functions. We provide necessary and sufficient conditions for the existence of a convex function $F\in C^{1,1}_{\textrm{loc}}(\mathbb{R}^n)$ such that $F=f$ and $\nabla F=G$ on $E$.
We give a useful explicit formula for such an extension $F$, and a variant of our main result for the class $C^{1, \omega}_{\textrm{loc}}$, where $\omega$ is a modulus of continuity. We also present two applications of these results, concerning how to find $C^{1,1}_{\textrm{loc}}$ convex hypersurfaces with prescribed tangent hyperplanes on a given subset of $\mathbb{R}^n$, and some explicit formulas for (not necessarily convex) $C^{1,1}_{\textrm{loc}}$ extensions of $1$-jets.
\end{abstract}

\maketitle

\section{Introduction and main results}

In \cite{AzagraMudarra1, AzagraLeGruyerMudarra, AzagraMudarra2} we considered the following problem.
\begin{prob}\label{main problem}
If $\mathcal{C}$ is a class of differentiable functions on $\R^n$ and we are given a subset $E$ of $\R^n$ and two functions $f:E\to\R$ and $G:E\to\R^n$, how can we decide whether there is a {\em convex} function $F\in\mathcal{C}$ such that $F(x)=f(x)$ and $\nabla F(x)=G(x)$ for all $x\in E$? 
\end{prob}
In those articles the problem was solved when $\mathcal{C}$ is one of the classes $C^{1,1}(\R^n)$, $C^{1, \omega}(\R^n)$, $C^{1}(\R^n)$.
We refer to the introductions of the papers \cite{AzagraLeGruyerMudarra, AzagraMudarra1, AzagraMudarra2, GhomiJDG2001, MinYan} for some motivation and background for this problem. We also recommend to see \cite{BierstoneMilmanPawlucki1, BierstoneMilmanPawlucki2, BrudnyiShvartsman, Fefferman2005, Fefferman2006, FeffermanSurvey, FeffermanIsraelLuli3, FeffermanIsraelLuli1, FeffermanIsraelLuli2, FeffermanKlartag, FeffermanShvartsman, Glaeser, HajlaszEtALSobolevExtension, KlartagZobin, KoskelaEtAl, LeGruyer1, PinamontiSpeightZimmerman, Shvartsman2010, Shvartsman2014, Shvartsman2017, ShvartsmanZobin, Wells, Whitney, Zobin1998, Zobin1999} and the references therein for information about general (we mean not necessarily convex) Whitney extension problems for jets and for functions.

Nothing is known about Problem \ref{main problem} in the case that $E$ is arbitrary and $\mathcal{C}=C^{m}$, $m\geq 2$, and in fact the problem looks extremely hard to solve for higher order differentiability classes, in view of the following two facts: 1) partitions of unity cannot be used to patch local convex extensions, as they  destroy convexity; and 2) convex envelopes do not preserve smoothness of orders higher than $C^{1,1}$, so the techniques of \cite{AzagraLeGruyerMudarra, AzagraMudarra1, AzagraMudarra2} cannot be employed to construct $C^2$ extensions of jets. See \cite{AzagraMudarra3} for the special case that $E$ is convex and $m=\infty$.

In this paper we study and solve Problem \ref{main problem} for the class $C^{1,1}_{\textrm{loc}}(\R^n)$ of differentiable functions with locally Lipschitz gradients (see Section 3 below for a more precise definition including its natural topological structure as a Fr\'echet space). This class of functions is very interesting at least for the following two reasons. On the one hand, $C^{1,1}_{\textrm{loc}}$ regularity is good enough for many purposes in  Real Analysis and Differential Geometry. On the other hand, in contrast with the rather sparse class of (globally) $C^{1,1}$ functions, the class $C^{1,1}_{\textrm{loc}}$ comprises lots of convex functions. Indeed, for instance the function $f(x)=x^4$, $x\in\R$, does not have a Lipschitz derivative, hence it is not of class $C^{1,1}$, but it is of course of class $C^{1,1}_{\textrm{loc}}$. Much more generally: as a consequence of the main result of \cite{Azagra2013}, every convex function on $\R^n$ can be uniformly approximated by $C^{1,1}_{\textrm{loc}}$ convex functions; however, this is not true if we replace the class  $C^{1,1}_{\textrm{loc}}$ with the class $C^{1,1}$ (any function which grows more than quadratically at infinity serves as a couterexample). Further motivation for the present paper comes from the work \cite{AzagraHajlasz}, where we need to know when we can find convex extensions of class $C^{1,1}_{\textrm{loc}}(\R^n)$ of a given jet $(f, G):E\to \R\times\R^n$. There are also other interesting applications of solutions to Problem \ref{main problem}; see Section 4 below.

Due to the mentioned fact that partitions of unity are useless in this kind of problems, the $C^{1,1}$ convex extension results of \cite{AzagraLeGruyerMudarra} do not give us a method of deciding whether or not a given jet has a $C^{1,1}_{\textrm{loc}}$ convex extension. 
As a matter of fact, there are very important differences between the global behavior of $C^{1,1}_{\textrm{loc}}$ convex functions and that of $C^{1,1}$ convex functions. Those differences may even be decisive in determining whether a given jet has extensions in these classes. For instance, $C^{1,1}$ convex functions on $\R^n$ cannot have what in \cite{AzagraMudarra2} we called {\em corners at infinity}, but $C^{1,1}_{\textrm{loc}}$ convex functions (and even real-analytic convex functions) can have them. 
Neither can the results of \cite{AzagraMudarra2} be applied to solve Problem \ref{main problem} for $\mathcal{C}=C^{1,1}_{\textrm{loc}}$. This is due both to the unsuitability of the conditions of the main result of \cite{AzagraMudarra2} (which ignore the difficulty that, in addition to corners, $C^{1,1}_{\textrm{loc}}$ convex functions may have other kinds of weaker singularities at infinity, such as what we could call {\em H\"older wedges at infinity}; see Examples \ref{example of Hoder wedge at infinity}, \ref{example of jet with Holder wedge and no extension} and \ref{example of jet with Holder wedge and many extensions} below), and also to some important elements of its proof.
In order to solve Problem \ref{main problem}, in this paper we will make a hybrid of results and methods of \cite{AzagraLeGruyerMudarra} and \cite{AzagraMudarra2}, also using some ideas of \cite{Azagra2013} and \cite{KirchheimKristensen}.

As in \cite{AzagraMudarra2}, our most general results contain some complicated conditions which may be difficult to grasp at first reading. For this reason, and in order to facilitate understanding of this paper, we will start by examining some corollaries and examples. It will also be convenient to state the following reformulation of the main result of \cite{AzagraLeGruyerMudarra}.

\begin{thm}[Azagra-LeGruyer-Mudarra]\label{reformulation from AzagraLeGruyerMudarra main result}
Let $E$ be an arbitrary nonempty subset of $\R^n$. Let $f:E\to\R$, $G:E\to\R^n$ be given functions. Assume that there exists some $M>0$ such that
\begin{equation}\label{every tangent parabola lies above all tangent planes}
f(z)+\langle G(z), x-z\rangle\leq f(y)+\langle G(y), x-y\rangle +\frac{M}{2}|x-y|^2
\end{equation}
for every $y, z\in E$ and every $x\in \R^n$. Then the formula  
\begin{equation}\label{extension formula case C11 conv}
F=\textrm{conv}\left(x\mapsto \inf_{y\in E}\left\{f(y)+\langle G(y), x-y\rangle+\frac{M}{2}|x-y|^2\right\}\right)
\end{equation}
defines a $C^{1,1}$ convex extension of $f$ to $\R^n$ which satisfies $\nabla F=G$ on $E$ and $\textrm{Lip}(\nabla F)\leq M$.
 
Conversely, if there is a $C^{1,1}(\R^n)$ convex extension $F$ of the $1$-jet $(f, G)$, then \eqref{every tangent parabola lies above all tangent planes} must be satisfied for every $M\geq\textrm{Lip}(\nabla F)$.
\end{thm}
Here $\textrm{conv}(g)$ denotes the convex envelope of a function $g$, that is, 
\begin{equation}\label{defn of convex envelope}
\textrm{conv}(g)(x)=\sup\{\varphi(x) \, : \, \varphi \textrm{ is convex, } \varphi\leq g\}.
\end{equation}
Other useful expressions for $\textrm{conv}(g)$ are given by
\begin{equation}\label{convex envelope as the inf...}
\textrm{conv}(g)(x)=\inf\left\lbrace \sum_{j=1}^{n+1}\lambda_{j} g(x_j) \, : \, \lambda_j\geq 0,
\sum_{j=1}^{n+1}\lambda_j =1, \, x=\sum_{j=1}^{n+1}\lambda_j x_j \right\rbrace
\end{equation}
(see \cite[Corollary 17.1.5]{Rockafellar} for instance),
and by the Fenchel biconjugate of $g$, that is,
\begin{equation}
\textrm{conv}(g)=g^{**},
\end{equation}
where 
\begin{equation}
h^{*}(x):=\sup_{v\in \R^n}\{\langle v, x\rangle -h(v)\}
\end{equation}
(see \cite[Proposition 4.4.3]{BorweinVanderwerffbook} for instance). 

Theorem \ref{reformulation from AzagraLeGruyerMudarra main result} is not explicitly stated in \cite{AzagraLeGruyerMudarra}, but it is implicitly contained in the proof of \cite[Theorem 2.4]{AzagraLeGruyerMudarra}.
Geometrically speaking, the epigraph of $F$ is the closed convex envelope in $\R^{n+1}$ of the union of the family of paraboloids $\{\mathcal{P}_{y}: y\in E\}$, where $\mathcal{P}_y=\{(x,t)\in\R^n\times\R : t=f(y)+\langle G(y), x-y\rangle +\frac{M}{2}|x-y|^2, x\in\R^n\}$, and condition \eqref{every tangent parabola lies above all tangent planes} tells us that these paraboloids must lie above the putative tangent hyperplanes $\{(x,t)\in\R^n\times\R : t=f(z)+\langle G(z), x-z\rangle\}$.

In this paper we will be looking for analogues of this result for the more complicated case of $C^{1,1}_{\textrm{loc}}$ convex extensions of $1$-jets. If the given jet $(f, G)$ has the property that $\textrm{span}\{G(y)-G(z) : y, z\in E\}=\R^n$, then our main result is easier to understand and use, and can be stated as follows.

\begin{thm}\label{Corollary 2 to the main coercive result}
Let $E$ be an arbitrary nonempty subset of $\R^n$. Let $f:E\to\R$, $G:E\to\R^n$ be functions such that 
\begin{equation}\label{essentially coercive data corollary 2}
\textrm{span}\{G(x)-G(y) : x, y\in E\}=\R^n.
\end{equation}
Then there exists a convex function $F\in C^{1,1}_{\textrm{loc}}(\R^n)$ such that $F_{|_E}=f$ and $(\nabla F)_{|_E}=G$ if and only if for each $k\in \N$ there exists a number $A_k\geq 2$ such that
\begin{equation}\label{every tangent function lies above all tangent planes corollary 2}
f(z)+\langle G(z), x-z\rangle\leq f(y)+\langle G(y), x-y\rangle +\frac{A_k}{2}|x-y|^2 
\end{equation}
$$
\textrm{ for every }z\in E,
\textrm{ every }y\in E\cap B(0, k), \textrm{ and every } x\in B(0, 4k).
$$
Furthermore, if $G$ is bounded then a formula for such an extension $F$ is given by
\begin{equation}\label{extension formula for G bounded 1}
F(x)=\textrm{conv}\left( x\mapsto \inf_{y\in E}\{f(y)+\langle G(y), x-y\rangle +\frac{1}{2}\left( A_{k(y)} +4\|G\|_{\infty}+1\right) |x-y|^2\}\right),
\end{equation}
where $k(y)$ is defined as the first positive integer such that $y\in B(0, k)$, and $\|G\|_{\infty}:=\sup_{x\in E}|G(x)|$.
\end{thm}

In the above theorems, as in the rest of the paper, $B(x, r)$ denotes the closed ball of center $x$ and radius $r$.

Of course inequality \eqref{every tangent function lies above all tangent planes corollary 2} reminds us of \eqref{every tangent parabola lies above all tangent planes}, but we should also note a very important difference between these conditions, as well as the asymmetry of the new condition \eqref{every tangent function lies above all tangent planes corollary 2}. Namely,
in condition \eqref{every tangent parabola lies above all tangent planes} both $z$ and $y$ run in all of the set $E$, while in condition \eqref{every tangent function lies above all tangent planes corollary 2} the point $z$ runs in all of $E$ but the location of the point $y$ is restricted to the intersection of $E$ with the ball $B(0, k)$, and the point $x$ is only asked to be in the ball $B(0, 4k)$, as opposed to all of $\R^n$ (and, of course, the condition's constant $A_k$ depends on $k$). Thus one can say that condition \ref{every tangent function lies above all tangent planes corollary 2} is global on the left-hand side of the inequality, but semi-global on the right-hand side (this explains our previous use of the term {\em asymmetry}). Furthermore, let us emphasize that condition \eqref{every tangent function lies above all tangent planes corollary 2} is {\em not equivalent} to saying that the restriction of the jet $(f, G)$ to the set $E\cap B(0,k)$ satisfies condition \eqref{every tangent parabola lies above all tangent planes} for each $k\in \N$.

Now let us proceed to study the general situation where we do not necessarily have $\textrm{span}\{G(x)-G(y) : x, y\in E\}=\R^n$. In this case, as we saw in \cite{AzagraMudarra2}, the possible presence of {\em corners at infinity}, makes things more complicated. If we are seeking $C^{1,1}_{\textrm{loc}}$ convex extensions, then we have to be even more careful: not only do we have to deal with such corners at infinity, but also with what we could call {\em H\"older wedges at infinity}, a terminology which is certainly vague and we do not intend to make precise but may become intuitively clear after having a look at the following examples.

\begin{ex}\label{example of Hoder wedge at infinity}
{\em Let $f, g:\R^2\to\R$ be defined by $$f(x,y)=\sqrt{ |x|^3 +e^{-2y}}, \quad \textrm{ and } \quad  g(x,y)=|x|^{3/2}.$$ Both are convex functions, and $f\in C^{1,1}_{\textrm{loc}}$, but $g\in C^{1, 1/2}_{\textrm{loc}}\setminus C^{1,1}_{\textrm{loc}}$. However, we have that $f\geq g$ and $\lim_{y\to\infty}f(x,y)=g(x, y)$.  We are tempted to say that $g$ is a H\"older wedge that supports $f$ at infinity. }
\end{ex}

\begin{ex}\label{example of jet with Holder wedge and no extension}
{\em Let $g(x, y)=|x|^{3/2}$, $(x,y) \in \R^2$. Let $E=\{(x, y)\in\R^2 : |x|\geq \min\{1, e^{y}\}\}$, and define $f$ and $G$ on $E$ by
$$
f=g \quad \text{on} \quad E,
$$
and
$$
G(x,y)=\nabla g(x,y)= \left(\frac{3}{2}|x|^{1/2}\textrm{sign}(x), 0\right) \quad \text{ if } (x,y) \in E.
$$
Then there is no convex function $F\in C^{1,1}_{\textrm{loc}}$ such that $F=f$ and $\nabla F=G$ on $E$, because for every convex function $\varphi:\R^2\to\R$ such that $\varphi=f$ on $E$ we must have $\varphi(x,y)=|x|^{3/2}$ on $\R^2$. As a matter of fact, for every pair of $C^1$ convex functions $\psi:\R^2\to\R$ and $\eta:\R\to\R$, we have that if $\psi(x,y)=\eta(x)$ for all $(x,y)\in E$ then $\psi(x,y)=\eta(x)$ for all $(x,y)\in\R^2$. Let us prove this assertion. We first claim that for every $(x_0, y_0)\in \R^2$ we have $\frac{\partial\psi}{\partial y}(x_0, y_0)=0$. Indeed, by convexity we have 
$$
\psi(x,y)\geq \psi(x_0, y_0)+a(x-x_0)+b(y-y_0)
$$
for all $(x,y)\in\R^2$, where we denote $\nabla\psi(x_0, y_0)=(a, b)$. Taking $(x, y)$ of the form $(x(t), y(t))=(2, t)$, $t\in\R$, and noting that $(2,t)\in E$ for all $t\in\R$, we obtain
$$
\eta(2)=\psi(2,t)\geq \psi(x_0, y_0)+a(2-x_0)+b(t-y_0)
$$
for all $t\in\R$, which is impossible unless $b=0$. So we have that $\frac{\partial \psi}{\partial y}(x,y)=0$ for all $(x,y)\in\R^2$, and therefore, for each $x\in\R$, the function $\R\ni y\mapsto \psi(x,y)\in\R$ does not depend on $y$. Since for every $(x,y)\in \R^2$ with $x\neq 0$ there exists some $y_0$ with $(x,y_0)\in E$, we deduce that $\psi(x,y)=\psi(x,y_0)=\eta(x)$. Thus $\psi(x,y)=\eta(x)$ for all $(x,y)\in\R^2$ with $x\neq 0$, hence by continuity also for all $(x,y)\in\R^2$.

Note that this example also shows that there are jets $(f, G)$ on $E$ such that: 1) they have $C^{1}$ convex extensions (even of class $C^{1, 1/2}$) to all of $\R^n$; 2) their restrictions to $E\cap B(0,k)$ satisfy condition $(CW^{1,1})$ of \cite{AzagraMudarra1, AzagraLeGruyerMudarra} (and in particular Whitney's condition for $C^{1,1}$ extension too) for each $k\in\N$; 3) and yet they do not have $C^{1,1}_{\textrm{loc}}$ convex extensions to all of $\R^n$. We thus see that there are global effects that may become very selective to prevent or admit the existence of convex extensions of a given jet in various differentiability classes.}
\end{ex}

\begin{ex}\label{example of jet with Holder wedge and many extensions}
{\em Let $g(x, y)=|x|^{3/2}$, $(x,y) \in \R^2$. Let $E=\{(x, y)\in\R^2 : |x|\geq e^{y}\}$, and define $f$ and $G$ on $E$ by
$$
f=g \quad \text{on} \quad E,
$$
and
$$
G(x,y)=\nabla g(x,y)= \left(\frac{3}{2}|x|^{1/2}\textrm{sign}(x), 0\right) \quad \text{ if } (x,y) \in E.
$$
We claim that there exist many convex functions $F\in C^{1,1}_{\textrm{loc}}$ such that $F=f$ and $\nabla F=G$ on $E$. It is not easy to give a direct proof of this assertion without applying Theorem \ref{Corollary 2 to the main coercive result} on a new, larger set $E$. We just note that this is a consequence of our next result (see the proof of Proposition \ref{there cannot be any extension method with a good control of the seminorms}(1) in Section 3 below for a detailed construction of a similar example). However, all of such extensions $F$ will be supported by a H\"older wedge at infinity, in the sense that for every $x\in\R\setminus\{0\}$ and every $y\leq \log |x|$ we have $F(x, y)=|x|^{3/2}$, and in particular $\lim_{y\to-\infty}F(x,y)=|x|^{3/2}$, and also $F(x, y)\geq |x|^{3/2}$ for all $(x,y)\in\R^2$. }
\end{ex}

\begin{rem}
{\em Let us emphasize the essential difference between Example \ref{example of jet with Holder wedge and no extension} and \ref{example of jet with Holder wedge and many extensions}. Thanks to the geometrical differences between the domains of the jets, in the second of these examples it is possible to add one more point and one more jet to our problem so as to obtain a new extension problem which can be solved by applying Theorem \ref{Corollary 2 to the main coercive result}, while in the first one this is impossible.
}
\end{rem}

\medskip

Before presenting our main theorem for the case that $\textrm{span}\{G(x)-G(y) : x, y\in E\}\neq\R^n$, we need a definition and a result from \cite{AzagraMudarra2, Azagra2013} which help us understand the global geometrical behavior of convex functions and provide us with a canonical representation that may be used to reduce problems about general convex functions to simpler problems about coercive convex functions.

\begin{defn}\label{defn essential coercivity}
{\em Let $Z$ be a Euclidean space, and $P:Z\to X$ be the orthogonal projection onto a subspace $X\subseteq Z$.
We will say that a function $f$ defined on a subset $S$ of $Z$ is {\em essentially $P$-coercive} provided that there exists a linear function $\ell:Z\to\R$ such that for every sequence $(x_k)_k\subset S$ with $\lim_{k\to\infty}|P(x_k)|=\infty$ one has
$$
\lim_{k\to\infty}\left(f-\ell\right)(x_k)=\infty.
$$
We will say that $f$ is {\em essentially coercive} whenever $f$ is essentially $I$-coercive, where $I:Z\to Z$ is the identity mapping.

For instance, a function $f:\R^n\to\R$ is {\em essentially coercive} provided there exists a linear function $\ell:\R^n\to\R$ such that 
$$
\lim_{|x|\to\infty}\left( f(x)-\ell(x)\right)=\infty.
$$

If $X$ is a linear subspace of $\R^n$, we will denote by $P_{X}:\R^n\to X$ the orthogonal projection, and we will say that $f:S\to \R$ is {\em coercive in the direction of $X$} whenever $f$ is $P_{X}$-coercive. 

We will denote by $X^{\perp}$ the orthogonal complement of $X$ in $\R^n$. For a subset $V$ of $\R^n$, $\textrm{span}(V)$ will stand for the linear subspace spanned by the vectors of $V$. 

We also recall that, for a convex function $f:\R^n\to\R$, the subdifferential of $f$ at a point $x\in\R^n$ is defined as
$$
\partial f(x)=\{\xi\in\R^n :  f(z)\geq f(x)+ \langle \xi, z-x\rangle \textrm{ for all } z\in\R^n\},
$$
and each $\xi\in\partial f(x)$ is called a subgradient of $f$ at $x$. }
\end{defn}

\begin{thm}\label{rigid global behavior of convex functions}[See the proofs of \cite[Theorem 1.11]{AzagraMudarra2} and \cite[Lemma 4.2]{Azagra2013}]
For every convex function $f:\R^n\to\R$ there exist a unique linear subspace $X_f$ of $\R^n$, a unique vector $v_{f}\in X_{f}^{\perp}$, and a unique essentially coercive function $c_{f}:X_f\to\R$ such that $f$ can be written in the form
$$
f(x)=c_{f}(P_{X_{f}}(x)) +\langle v_{f}, x\rangle, \quad x\in\R^n.
$$
The subspace $X_{f}$ coincides with $\textrm{span}\{u-w : u\in\partial f(x), w\in\partial f(y), x, y\in\R^n\}$, and the vector $v_f$ coincides with $Q_{X_f}(\xi_0)$ for any $\xi_0\in\partial f(x_0)$, $x_0\in\R^n$, where $Q_{X_f}=I-P_{X_f}$ is the orthogonal projection of $\R^n$ onto $X_{f}^{\perp}$.
Moreover, if $Y$ is a linear subspace of $\R^n$ such that $f$ is essentially coercive in the direction of $Y$, then $Y\subseteq X_f$.
\end{thm}
The above characterization of $X_f$ and $v_f$ does not appear in the statement of \cite[Theorem 1.11]{AzagraMudarra2}, but it is implicit in its proof.

Now we are ready to state the most important result of this paper.

\begin{thm}\label{First variant of MainTheorem with P}
Given an arbitrary nonempty subset $E$ of $\R^n$, a linear subspace $X\subset\R^n$, the orthogonal projection $P:=P_{X}:\R^n\to X$, and two functions $f:E \to \R, \: G: E \to \R^n$, the following is true. There exists a convex function $F:\R^n\to\R$ of class $C^{1,1}_{\textrm{loc}}$ such that $F_{|_E}=f$, $(\nabla F)_{|_E} =G$, and $X_{F}=X$, if and only if the following conditions are satisfied.
\begin{itemize}
\item[$(i)$] $Y:=\textrm{span}\left(\{G(y)-G(z) : y, z\in E\}\right)\subseteq X$.
\item[$(ii)$] If $\ell:=\dim Y< d:= \dim X$, then there exist points $p_1, \ldots, p_{d-\ell} \in \R^n \setminus E$, numbers $\beta_1,\ldots, \beta_{d-\ell} \in \R$, vectors $w_1, \ldots, w_{d-\ell} \in \R^n$, and a sequence of numbers $A_k\geq 2$, $k\in\N$, such that, 
denoting: $E^{*}:=E\cup\{p_1, ..., p_{d-\ell}\}$; $t_y:=f(y)$ and $\xi_y:=G(y)$ for $y\in E$; $t_y=\beta_j$ and $\xi_y=w_j$ for $y=p_j$, $j=1, ..., d-\ell$, we have that
\begin{equation}\label{completing the span of derivatives 2}
\textrm{span}\{\xi_y-\xi_z : y, z\in E^{*}\}=X,
\end{equation}
and
\begin{equation}\label{every tangent function lies above all tangent planes MainthmwithP 2}
t_z+\langle \xi_z, x-z\rangle  \leq t_y+\langle \xi_y, x-y\rangle +\frac{A_k}{2}|P(x-y)|^2 
\end{equation}
$$
\textrm{ for every } z\in E^{*}, y\in E^{*}\cap P^{-1}(B_X(0, k)), \, x\in P^{-1}(B_X(0, 4k)).
$$
\item[$(iii)$] If $\ell=d$, then the preceding condition holds with $E$ in place of $E^{*}$ (no need to add new data).
\end{itemize}

Furthermore, if $G$ is bounded then a formula for such an extension $F$ is given by
\begin{equation}
F(x)=\textrm{conv}\left( x\mapsto \inf_{y\in E^{*}}\{t_y+\langle \xi_y, x-y\rangle +\frac{1}{2}\left( A_{k(y)} +4\|G\|_{\infty}+1\right) |P(x-y)|^2\}\right)
\end{equation}
where $k(y)$ is defined as the first positive integer such that $y\in P^{-1}(B_X(0, k))$, and $\|G\|_{\infty}:=\sup_{x\in E}|G(x)|$.
\end{thm}

In particular, by considering the case that $X=\R^n$, we obtain a characterization of the $1$-jets which admit $C^1$ convex extensions that are essentially coercive on $\R^n$, thus improving Theorem \ref{Corollary 2 to the main coercive result} (which does not directly address situations like that of Example \ref{example of jet with Holder wedge and many extensions}).

The rest of this paper is organized as follows. In Section 2 we provide more technical and more general versions of the above results whose statements have the advantage of providing explicit formulas for the extensions $F$. In Section 3 we study the natural and important question whether or not  one can obtain $C^{1,1}_{\textrm{loc}}$ convex extensions whose gradients have local Lipschitz constants that can be controlled by the local Lipschitz constants of the gradients of the functions $\varphi_y$ appearing in the statement of Theorem \ref{Corollary to C11 Whitney thm for coercive convex functions} below (or equivalent, by the numbers $A_k$ in the statements of Theorems \ref{Corollary 2 to the main coercive result} and \ref{First variant of MainTheorem with P}). As we will see, and in sharp contrast to the $C^{1,1}$ case that we studied in \cite{AzagraLeGruyerMudarra}, neither our method of extension nor any other can achieve this. Nonetheless we also obtain some positive results for families of functions which are {\em uniformly essentially coercive} in an appropriate sense. In Section 4 we will present some applications of our results. Finally in Section 5 we give the proofs of the main theorems.

\medskip

\section{Technical versions of the main results, with explicit formulas}

In this section we give some versions of our main results which have the advantage of providing us with explicit formulas for the extension functions $F$ (and the disadvantage that their statements involve the existence of families of functions $\varphi_y$ which we are not told how to find). These technical versions of the main results will also help us understand their proofs better, splitting them into two parts which use different methods. In order to see how we can construct appropriate families of functions $\varphi_y$ that satisfy the assumptions of Theorem \ref{Corollary to C11 Whitney thm for coercive convex functions}  starting from condition \eqref{every tangent function lies above all tangent planes corollary 2} in Theorem \ref{Corollary 2 to the main coercive result}, see Section 5.3 below. 

Let us begin with the easier case that $\textrm{span}\{G(y)-G(z) : y, z\in E\}=\R^n$.

\begin{thm}\label{Corollary to C11 Whitney thm for coercive convex functions}
Let $E$ be an arbitrary nonempty subset of $\R^n$. Let $f:E\to\R$, $G:E\to\R^n$ be functions  such that 
\begin{equation}\label{essentially coercive data corollary}
\textrm{span}\{G(x)-G(y) : x, y\in E\}=\R^n.
\end{equation}
Then there exists a convex function $F\in C^{1,1}_{\textrm{loc}}(\R^n)$ such that $F_{|_E}=f$ and $(\nabla F)_{|_E}=G$ if and only if the following condition is satisfied. For each $y\in E$ there exists a (not necessarily convex) $C^{1,1}_{\textrm{loc}}$ function $\varphi_y:\R^n\to [0, \infty)$ such that:
\begin{equation}\label{phiy equals 0 at y}
\varphi_{y}(y)=0, \nabla\varphi_{y}(y)=0;
\end{equation}
\begin{equation}\label{finite sup of Lip constants of phiy on balls corollary}
M_R:=\sup\left\{\frac{|\nabla\varphi_{y}(x)-\nabla\varphi_{y}(z)|}{|x-z|} \, : x, z\in B(0, R), x\neq z,  \, y\in E\cap B(0,R)\right\} <\infty
\end{equation}
for every $R>0$, and
\begin{equation}\label{every tangent function lies above all tangent planes corollary}
f(z)+\langle G(z), x-z\rangle\leq f(y)+\langle G(y), x-y\rangle +\varphi_{y}(x) 
\end{equation}
for every $y, z\in E$ and every $x\in \R^n$.
Moreover, when these conditions are satisfied, the extension $F$ can be taken to be essentially coercive, and in fact, for every number $a>0$ the formula
\begin{equation}\label{formula for Fa in thm2}
F=F_a=\textrm{conv}\left(x\mapsto \inf_{y\in E}\left\{f(y)+\langle G(y), x-y\rangle+\varphi_{y}(x)+a|x-y|^2\right\}\right)
\end{equation}
defines such an essentially coercive $C^{1,1}_{\textrm{loc}}$ convex extension of the jet $(f, G)$ to $\R^n$.
\end{thm}

As for the most general situation that $\textrm{span}\{G(y)-G(z) : y, z\in E\}$ does not necessarily coincide with $\R^n$, we have the following technical version of Theorem \ref{First variant of MainTheorem with P}.

\begin{thm}\label{MainTheorem with P}
Given an arbitrary nonempty subset $E$ of $\R^n$, a linear subspace $X\subset\R^n$, the orthogonal projection $P:=P_{X}:\R^n\to X$, and two functions $f:E \to \R, \: G: E \to \R^n$, the following is true. There exists a convex function $F:\R^n\to\R$ of class $C^{1,1}_{\textrm{loc}}$ such that $F_{|_E}=f$, $(\nabla F)_{|_E} =G$, and $X_{F}=X$, if and only if the following conditions are satisfied.
\begin{itemize}
\item[$(i)$] $Y:=\textrm{span}\left(\{G(y)-G(z) : y, z\in E\}\right)\subseteq X$.
\item[$(ii)$] If $k:=\dim Y< d:= \dim X$, then there exist points $p_1, \ldots, p_{d-k} \in \R^n \setminus E$, numbers $\beta_1,\ldots, \beta_{d-k} \in \R$, and vectors $w_1, \ldots, w_{d-k} \in \R^n$ such that for every $y\in E^{*}:=E\cup\{p_1, ..., p_{d-k}\}$ there exists a (not necessarily convex) function $\varphi_{y}:X\to [0, \infty)$ of class $C^{1,1}_{\textrm{loc}}$ such that,
denoting: $t_y:=f(y)$ and $\xi_y:=G(y)$ for $y\in E$; $t_y=\beta_j$ and $\xi_y=w_j$ for $y=p_j$, $j=1, ..., d-k$, 
we have that:
\begin{equation}\label{completing the span of derivatives}
\textrm{span}\{\xi_y-\xi_z : y, z \in E^{*}\}=X;
\end{equation}
\begin{equation}\label{phiy equals 0 at y corollary}
\varphi_{y}(P(y))=0, \nabla\varphi_{y}(P(y))=0;
\end{equation}
\begin{equation}\label{finite sup of Lip constants of phiy on balls main thm}
\sup\left\{\frac{|\nabla\varphi_{y}(u)-\nabla\varphi_{y}(v)|}{|u-v|} \, : \, y\in E^{*}\cap P^{-1}(B_{X}(0, R)), u, v\in B_{X}(0, R), u\neq v \right\}<\infty
\end{equation}
for every $R>0$;
and
\begin{equation}\label{every tangent function lies above all tangent planes MainthmwithP}
t_z+\langle \xi_z, x-z\rangle \leq t_y +\langle \xi_y, x-y\rangle +\varphi_{y}(P(x)) 
\end{equation}
for every $z, y \in E^{*}$ and every $x\in \R^n$.
\item[$(iii)$] If $k=d$, then the preceding condition holds with $E$ in place of $E^{*}$ (no need to add new data).
\end{itemize}
Moreover, whenever these conditions are satisfied, for every number $a>0$ the formula
\begin{equation}\label{formula for F in main thm}
F=\textrm{conv}\left(x\mapsto \inf_{y\in E^{*}}\left\{t_y+ \langle\xi_y, x-y\rangle+\varphi_{y}(P(x))+a|P(x-y)|^2\right\}\right)
\end{equation}
defines a $C^{1,1}_{\textrm{loc}}$ convex extension of the jet $(f, G)$ to $\R^n$ which satisfies $X_F=X$.
\end{thm}

There are  analogues of all of the above results for the classes $C^{1, \alpha}_{\textrm{loc}}$ or $C^{1, \omega}_{\textrm{loc}}$, where $\omega$ is a concave, strictly increasing modulus of continuity with $\omega(\infty)=\infty$. It suffices to replace $|x|^2$ with $\theta(|x|)$, where $\theta(t):=\int_{0}^{t}\omega(s)ds$, and make some other obvious changes. For instance, we have the following version of Theorem \ref{MainTheorem with P} for the class $C^{1, \omega}_{\textrm{loc}}$.

\begin{thm}\label{MainTheorem with P for C1omegaloc}
Given an arbitrary nonempty subset $E$ of $\R^n$, a linear subspace $X\subset\R^n$, the orthogonal projection $P:=P_{X}:\R^n\to X$, and two functions $f:E \to \R, \: G: E \to \R^n$, the following is true. There exists a convex function $F:\R^n\to\R$ of class $C^{1, \omega}_{\textrm{loc}}$ such that $F_{|_E}=f$, $(\nabla F)_{|_E} =G$, and $X_{F}=X$, if and only if the following conditions are satisfied.
\begin{itemize}
\item[$(i)$] $Y:=\textrm{span}\left(\{G(y)-G(z) : y, z\in E\}\right)\subseteq X$.
\item[$(ii)$] If $k:=\dim Y< d:= \dim X$, then there exist points $p_1, \ldots, p_{d-k} \in \R^n \setminus E$, numbers $\beta_1,\ldots, \beta_{d-k} \in \R$,  and vectors $w_1, \ldots, w_{d-k} \in \R^n$ such that for every $y\in E^{*}:=E\cup\{p_1, ..., p_{d-k}\}$ there exists a (not necessarily convex) function $\varphi_{y}:X\to [0, \infty)$ of class $C^{1, \omega}_{\textrm{loc}}$ such that,
denoting: $t_y:=f(y)$ and $\xi_y:=G(y)$ for $y\in E$; $t_y=\beta_j$ and $\xi_y=w_i$ for $y=p_i$, $i=1, ..., d-k$,
we have that:
\begin{equation}\label{completing the span of derivatives C1womegaloc}
\textrm{span}\{\xi_y-\xi_z : y,z \in E^{*}\}=X;
\end{equation}
\begin{equation}\label{phiy equals 0 at y corollary C1womegaloc}
\varphi_{y}(P(y))=0, \nabla\varphi_{y}(P(y))=0;
\end{equation}
\begin{equation}\label{finite sup of Lip constants of phiy on balls corollary C1womegaloc}
\sup\left\{\frac{|\nabla\varphi_{y}(u)-\nabla\varphi_{y}(v)|}{\omega(|u-v|)} \, : \, y\in E^{*}\cap P^{-1}(B_{X}(0, R)), u, v\in B_{X}(0, R), u\neq v \right\}<\infty
\end{equation}
for every $R>0$; and
\begin{equation}\label{every tangent function lies above all tangent planes corollary C1womegaloc}
t_z+\langle \xi_z, x-z\rangle \leq t_y +\langle \xi_y, x-y\rangle +\varphi_{y}(P(x)) 
\end{equation}
for every $z, y \in E^{*}$ and every $x\in \R^n$.
\item[$(iii)$] If $k=d$, then the preceding condition holds with $E$ in place of $E^{*}$ (no need to add new data).
\end{itemize}
Moreover, whenever these conditions are satisfied, for every number $a>0$ the formula
$$
F=\textrm{conv}\left(x\mapsto \inf_{y\in E^{*}}\left\{t_y+ \langle\xi_y, x-y\rangle+\varphi_{y}(P(x))+a\,\theta\left(|P(x-y)|\right)\right\}\right)
$$
defines a $C^{1, \omega}_{\textrm{loc}}$ convex extension of the jet $(f, G)$ to $\R^n$ which satisfies $X_F=X$.
\end{thm}

Finally, let us mention that our methods also allow us to establish explicit formulas for $C^1$ convex extensions of jets. We only state the result for the easier case that $\textrm{span}\{G(y)-G(z) : y, z\in E\}=\R^n$, because the most general result of this kind for the class $C^1$ has an excessively complicated statement.\footnote{Even if we assume $E$ to be closed, in some situations we would have to find and add new jets not only at a finite number of points $p_j$, but also at every point $z$ of the possibly infinite set $\overline{P(E)}\setminus E$. Although the latter jets $\xi_z$, $z\in \overline{P(E)}\setminus E$ are uniquely determined, the associated functions $\varphi_z$ are not, and in any case the process to define them is laborious.}

\begin{thm}\label{MainTheorem with P=I for C1}
Let $E$ be a closed nonempty subset of $\R^n$. Let $f:E\to\R$, $G:E\to\R^n$ be continuous functions  such that 
\begin{equation}\label{essentially coercive data thm C1}
\textrm{span}\{G(x)-G(y) : x, y\in E\}=\R^n.
\end{equation}
Then there exists a convex function $F\in C^{1}(\R^n)$ such that $F_{|_E}=f$ and $(\nabla F)_{|_E}=G$ if and only if for every $y\in E$ there exists a (not necessarily convex) differentiable function $\varphi_y:\R^n\to [0, \infty)$ such that:
\begin{equation}\label{phiy equals 0 at y thm C1}
\varphi_{y}(y)=0,  \,\,\, \nabla\varphi_{y}(y)=0,
\end{equation}
and
\begin{equation}\label{every tangent function lies above all tangent planes corollary C1}
f(z)+\langle G(z), x-z\rangle\leq f(y)+\langle G(y), x-y\rangle +\varphi_{y}(x) 
\end{equation}
for every $y, z\in E$ and every $x\in \R^n$.
Moreover, when these conditions are satisfied, for every number $a>0$ the formula
$$
F=F_a=\textrm{conv}\left(x\mapsto \inf_{y\in E}\left\{f(y)+\langle G(y), x-y\rangle+\varphi_{y}(x)+a|x-y|^2\right\}\right)
$$
defines such a $C^{1}$ convex extension of the jet $(f, G)$ to $\R^n$.
\end{thm}

\medskip

\section{Some remarks on the local Lipschitz seminorms of the extensions}

Recall that $C^{1,1}(\R^n)$ denotes the set of all functions $\varphi:\R^n\to\R$ which are differentiable and such that $\nabla\varphi:\R^n\to\R^n$ is Lipschitz. This space is naturally equipped with the seminorm
$$
\rho_{1,1}(\varphi)=\sup_{x, y\in\R^n, x\neq y}\frac{|\nabla\varphi(x)-\nabla \varphi(y)|}{|x-y|}=\textrm{Lip}(\nabla \varphi),
$$
and if we distinguish and fix a point $x_0\in\R^n$ and define
$$
\|\varphi\|_{C^{1,1}}(\R^n)=|\varphi(x_0)|+ |\nabla\varphi(x_0)| + \rho_{1,1}(\varphi),
$$
then $\left(C^{1,1}(\R^n), \|\cdot\|_{C^{1,1}(\R^n)}\right)$ is a Banach space. Now, if $E$ is a nonempty subset of $\R^n$ and $(f, G):E\to\R\times\R^n$ is a $1$-jet, we can define the Whitney seminorm of $(f, G)$ by
\begin{multline*}
\rho^{W}_{E}(f,G):= \inf \lbrace M>0 \: : \: 
|f(x)-f(y)-\langle G(y), x-y\rangle|\leq \tfrac{1}{2} M|x-y|^2 \textrm{ and }
\\
 |G(x)-G(y)|\leq M|x-y| \textrm{ for all } x, y\in E \rbrace.
\end{multline*}
If we consider the sets
$$
\mathcal{J}^{1,1}(E)=\left\{(f, G):E\to\R\times\R^n \,\, | \, \, \exists\, H\in C^{1,1}_{\textrm{loc}}(\R^n) \textrm{ s.t. } (H, \nabla H)=(f, G) \textrm{ on } E\right\},
$$
and
$$
\mathcal{J}^{W(1,1)}(E)=\left\{(f, G):E\to\R\times\R^n \,\, | \,\, \rho^{W}_{E}(f,G)<\infty\right\},
$$
then Whitney's extension theorem tells us that
$$
\mathcal{J}^{1,1}(E)= \mathcal{J}^{W(1,1)}(E)
$$
and provides us with a linear extension operator
$$
\mathcal{J}^{W(1,1)}(E)\ni (f, G)\mapsto \mathcal{W}(f,G)\in 
C^{1,1}(\R^n)
$$
with the property that 
\begin{equation}\label{C(n) in Whitney}
\rho_{1,1}\left(\mathcal{W}(f,G)\right)\leq C(n) \rho^{W}_{E}(f,G),
\end{equation}
where $C(n)$ is a constant only depending on the dimension $n$.

For the cone of convex functions of class $C^{1,1}$ we can consider the functional
\begin{multline*}
\rho^{CW}_{E}(f,G):= \inf \left\{ M>0 \, : \right.\\ \hspace{0.5cm} 
f(z)+\langle G(z), x-z\rangle\leq f(y)+\langle G(y), x-y\rangle +\frac{M}{2}|x-y|^2 \textrm{ for all } y, z\in E, \, \,\left.  x\in\R^n \right\},
\end{multline*}
and define the sets
$$
\mathcal{J}^{1,1}_{\textrm{conv}}(E)= \left\{(f, G):E\to\R\times\R^n \,\, | \, \, \exists\, H\in C^{1,1}_{\textrm{conv}}(\R^n) \textrm{ such that } (H, \nabla H)=(f, G) \textrm{ on } E\right\},
$$
and
$$
\mathcal{J}^{CW(1,1)}(E)=\left\{(f, G):E\to\R\times\R^n \,\, | \,\, \rho^{CW}_{E}(f,G)<\infty\right\}.
$$
The main results of \cite{AzagraMudarra1, AzagraLeGruyerMudarra} tell us that
$$
\mathcal{J}^{1,1}_{\textrm{conv}}(E)= \mathcal{J}^{CW(1,1)}(E)
$$
and show that the operator $(f, G)\mapsto F$ given by formula \eqref{extension formula case C11 conv} has the property that
$$
\rho_{1,1}\left(F\right)\leq A \, \rho^{CW}_{E}(f,G),
$$
where $A$ is an absolute constant (in fact we can take $A=1$). We also saw in \cite{AzagraLeGruyerMudarra} that a similar operator $\mathcal{E}$ for the problem of extending $1$-jets by (not necessarily convex) functions of class $C^{1,1}(\R^n)$ also has the property that
$$
\rho_{1,1}\left(\mathcal{E}(f,G)\right)\leq A\, \rho^{W}_{E}(f,G),
$$
where $A$ is an absolute constant (here one can take $A=7$).
In this respect this operator $\mathcal{E}$ behaves even better than the classical Whitney extension operator, because one has $\lim_{n\to\infty}C(n)=\infty$ in \eqref{C(n) in Whitney}. On the other hand, Whitney's operator is linear, while the one provided by \cite{AzagraLeGruyerMudarra} is not.

In this section we will see how this scenery changes dramatically when we consider $C^{1,1 \, \textrm{loc}}_{\textrm{conv}}(\R^n)$, the cone of convex functions which are of class $C^{1,1}_{\textrm{loc}}(\R^n)$, instead of the much smaller cone $C^{1,1}_{\textrm{conv}}(\R^n)$.
But first we must specify a natural topology in the space $C^{1,1}_{\textrm{loc}}(\R^n)$. Fixing a point $x_0\in\R^n$, we consider, for each $k\in\N$, the seminorm $\rho_k :C^{1,1}_{\textrm{loc}}(\R^n)\to [0, \infty)$ defined by
$$
\rho_k(\varphi)=\sup_{x, y\in B(x_0,k), x\neq y}\frac{|\nabla\varphi(x)-\nabla \varphi(y)|}{|x-y|}=\textrm{Lip}\left(\nabla \varphi_{|_{B(x_0,k)}}\right),
$$
and for $k=0$ we set
$$
\rho_0(\varphi)=|\varphi(x_0)|+|\nabla\varphi(x_0)|.
$$
Then it is not difficult to check that $C^{1,1}_{\textrm{loc}}(\R^n)$, equipped with the family of seminorms $\{\rho_k\}_{k\in\N\cup\{0\}}$, is a Fr\'echet space. A natural metric in this space is given by 
$$
\rho(\varphi, \psi)=\max_{k}\frac{2^{-k}\rho_k(\varphi-\psi)}{1+\rho_k(\varphi-\psi)}.
$$
In particular, a sequence $\{\varphi_j\}_{j\in\N}$ converges to $\varphi$ in $C^{1,1}_{\textrm{loc}}(\R^n)$ if and only if $$\lim_{j\to\infty}\rho_k (\varphi_j-\varphi)=0$$ for every $k\geq 0$. And a set $\mathcal{A}\subset C^{1,1}_{\textrm{loc}}(\R^n)$ is bounded if and only if for every $k$ the seminorm $\rho_k$ is bounded on $\mathcal{A}$. Boundedness of a set $\mathcal{A}$ in this space is often very useful, as it allows us, through the use of Arzel\`a-Ascoli's theorem and a diagonal argument, to extract a sequence $(\varphi_k)$ from $\mathcal{A}$ which converges, uniformly on bounded sets, to some function $\varphi\in C^{1,1}_{\textrm{loc}}(\R^n)$ (and such that $(\nabla\varphi_k)$ converges, uniformly on bounded sets, to $\nabla\varphi$).

Now, for any subset $E$ of $\R^n$, let us denote
$$
\mathcal{J}^{1,1}_{\textrm{loc}}(E):= \\
\left\{(f, G):E\to\R\times\R^n \, | \, \exists\, H\in C^{1,1}_{\textrm{loc}}(\R^n) \textrm{ such that } (H, \nabla H)=(f, G) \textrm{ on } E\right\},
$$
and its subset
$$
\mathcal{J}^{1,1 \, \textrm{loc}}_{\textrm{conv}}(E):= \\
\left\{(f, G):E\to\R\times\R^n \, | \, \exists\, H\in C^{1,1 \, \textrm{loc}}_{\textrm{conv}}(\R^n) \textrm{ such that } (H, \nabla H)=(f, G) \textrm{ on } E\right\}.
$$
On the set of $1$-jets on $E$ we may consider, for each $k\in\N$, the Whitney seminorms
\begin{multline*}
\rho^{W}_{k, E}(f,G)=
\inf_{M>0} \lbrace |f(x)-f(y)-\langle G(y), x-y\rangle|\leq \tfrac{1}{2}M|x-y|^2, \\ |G(x)-G(y)|\leq M|x-y| \,\,\, \forall x, y\in E\cap B(x_0,k) \rbrace,
\end{multline*}
and for $k=0$
$$
\rho^{W}_{0, {\Bl E}}(f, G)=|f(x_0)|+|G(x_0)|,
$$
where $x_0\in E$ is some fixed distinguished point,
and the metric
$$
\rho_{E}^{W}(\varphi, \psi)=
\max_{k}\frac{2^{-k}\rho_{k, E}^{W}(\varphi-\psi)}{1+\rho_{k, E}^{W}(\varphi-\psi)}.
$$
Again, Whitney's extension technique gives us
$$
\mathcal{J}^{1,1}_{\textrm{loc}}(E)=\{(g, G):E\to\R\times\R^n \, | \, \rho_{k, E}(f,G)<\infty \textrm{ for every } k\in\N\}.
$$
It is also well known that Whitney's extension operator
$$
\mathcal{J}^{1,1}_{\textrm{loc}}(E)\ni (f, G)\mapsto \mathcal{W}(f,G)\in 
C^{1,1}_{\textrm{loc}}(\R^n)
$$
is linear and continuous with respect to the metrics that we have defined in these spaces. This is equivalent to saying that if $\{(f_j, G_j)\}_{j\in\N}$ is a sequence in $\mathcal{J}^{1,1}_{\textrm{loc}}(E)$ such that $\{\rho_{k, E}^{W}(f_j, G_j)\}_{j\in\N}$ is bounded for every $k\geq 0$ then $\{\rho_{k}(\mathcal{W}(f_j, G_j))\}_{j\in\N}$ is also bounded for every $k\geq 0$.

In the framework of the problem that we are considering in this paper, we may
consider the following functionals
\begin{multline*}
\rho^{CW}_{k, E}(f,G):= \inf \lbrace M>0 : 
f(z)+\langle G(z), x-z\rangle\leq f(y)+\langle G(y), x-y\rangle +\frac{M}{2}|x-y|^2 \\ \forall y, z\in E\cap B(x_0,k) \, \,  x\in\R^n \rbrace,
\end{multline*}
and
$$
\rho^{CW}_{0, E}(f,G)=|f(x_0)|+|G(x_0)|,
$$
where $x_0$ is a fixed distinguished point of $E$,
and also (more naturally in our setting, in view of Theorem \ref{MainTheorem with P}, and using the notation of this result) the functionals
$$
\mu_{k, E}(f,G):=
\inf\left\{  \sup\left\{\frac{|\nabla\varphi_{y}(u)-\nabla\varphi_{y}(v)|}{|u-v|} \right\}\right\},
$$
where the supremum is taken over all points $y\in E^{*}\cap P^{-1}(B_{X}(P(x_0), k)), u, v\in B_{X}(P(x_0), k), u\neq v$, and the infimum is taken over all families of functions $\{\varphi_y\}$ satisfying the conditions of Theorem \ref{MainTheorem with P}. We also set
$$
\mu_{0, E}(f,G)=|f(x_0)|+|G(x_0)|.
$$
It is then natural to ask: does there exist a (not necessarily linear) extension operator
$$
\mathcal{J}^{1,1 \, \textrm{loc}}_{\textrm{conv}}(E)\ni (f, G)\mapsto \mathcal{E}(f,G)\in 
C^{1,1 \, \textrm{loc}}_{\textrm{conv}}(\R^n)
$$
such that, if $\{(f_j, G_j)\}_{j\in\N}$ is a sequence in $\mathcal{J}^{1,1 \, \textrm{loc}}_{\textrm{conv}}(E)$ so that $\{\rho^{CW}_{k, E}(f_j, G_j)\}_{j\in\N}$ is bounded for every $k\in\N$, then $\{\rho_{k}(\mathcal{E}(f_j, G_j))\}_{j\in\N}$ is bounded for every $k\in\N$ too? And more importantly, does there exist a (not necessarily linear) extension operator
$$
\mathcal{J}^{1,1 \, \textrm{loc}}_{\textrm{conv}}(E)\ni (f, G)\mapsto \mathcal{E}(f,G)\in 
C^{1,1 \, \textrm{loc}}_{\textrm{conv}}(\R^n)
$$
such that, if $\{(f_j, G_j)\}_{j\in\N}$ is a sequence in $\mathcal{J}^{1,1 \, \textrm{loc}}_{\textrm{conv}}(E)$ so that $\{\mu_{k, E}(f_j, G_j)\}_{j\in\N}$ is bounded for every $k\in\N$, then $\{\rho_{k}(\mathcal{E}(f_j, G_j))\}_{j\in\N}$ is bounded for every $k\in\N$ too?

Next we answer these questions in the negative.

\begin{prop}\label{there cannot be any extension method with a good control of the seminorms}
There exist a closed subset $E$ of $\R^2$ and a sequence of $1$-jets $\{(f_j, G_j)\}_{j\in\N}$ on $E$ such that:
\begin{enumerate}
\item There exists a sequence $\{F_j\}_{j\in\N}\subset C^{1,1 \, \textrm{loc}}_{\textrm{conv}}(\R^n)$ such that $(F_j, \nabla F_j)_{|_{_E}}=(f_j, G_j)$ for all $j\in\N$.
\item For every $k\in\N\cup\{0\}$ we have that $\sup_{j\in\N}\rho^{W}_{k, E}(f_j, G_j)<\infty$, \, $\sup_{j\in\N}\rho^{CW}_{k, E}(f_j, G_j)<\infty$, \, and $\sup_{j\in\N}\mu_{k, E}(f_j, G_j)<\infty$.
\item For every sequence $\{H_j\}_{j\in\N}\subset C^{1,1, \, \textrm{loc}}_{\textrm{conv}}(\R^n)$ such that $(H_j, \nabla H_j)_{|_{_E}}=(f_j, G_j)$ for all $j$, we have that \, $\sup_{j\in\N}\rho_{k}(H_j)=\infty$ for some $k\geq 1$.
\end{enumerate}
\end{prop}
\begin{proof}
Let 
$$
E:=E_1\cup E_2,
$$
where
$$
E_1=\{(x,y)\in\R^2 \, : \, |x|\geq e^{y}\} \, \textrm{ and } \, E_2= \{(x,y)\in\R^2 \, : \, |x|=1, y\in\N\},
$$ 
and define the sequence of $1$-jets $(f_j, G_j):E\to\R\times\R^2$ by
$$
f_j(x,y)=
\begin{cases}
|x| & \textrm{ if } (x,y)\in E_1 \\
1 & \textrm{ if } (x, y)\in E_2, 1\leq y\leq j+1 \\
2(y-j-1) & \textrm{ if } (x, y)\in E_2, y>j+1
\end{cases}
$$
and
$$
G_j(x,y)=
\begin{cases}
(-1,0) & \textrm{ if } (x,y)\in E_1, x<0 \\
(1,0) & \textrm{ if } (x,y)\in E_1, x>0 \\
(-1,0) & \textrm{ if } (x,y)\in E_2, x<0, 1\leq y\leq j+1 \\
(1,0) & \textrm{ if } (x,y)\in E_2, x>0, 1\leq y\leq j+1 \\
(0, 2) & \textrm{ if } (x,y)\in E_2, y>j+1.
\end{cases}
$$
Note that
\begin{equation}\label{expression for putative tangent planes in counterexample}
f(u,v)+\langle G(u,v), (x-u, y-v)\rangle=
\begin{cases}
-x & \textrm{ if } (u,v)\in E_1, u<0 \\
x & \textrm{ if } (u,v)\in E_1, u>0 \\
-x & \textrm{ if } (u,v)\in E_2, u<0, 1\leq v\leq j+1 \\
x & \textrm{ if } (u,v)\in E_2, u>0, 1\leq v\leq j+1 \\
2(y-j-1) & \textrm{ if } (u,v)\in E_2, v>j+1,
\end{cases}
\end{equation}
and in particular
\begin{equation}
m_{j}(x,y) :=\sup_{(u,v)\in E}\{f(u,v)+\langle G(u,v),(x, y)-(u,v)\rangle\}=
\max\left\{|x|, 2(y-j-1)\right\},
\end{equation}
To prove $(1)$ we are going to use Theorem \ref{Corollary to C11 Whitney thm for coercive convex functions}: we seek, for each $j\in\N$, a suitable family of functions $\{\varphi_{(j,u,v)}\}_{(u, v)\in E}$ of the form\footnote{It will be possible to find a family of quadratic functions $\{\varphi_{(j,u,v)}\}_{(u,v)\in E}$ satisfying the assumptions of Theorem \ref{Corollary to C11 Whitney thm for coercive convex functions} because $m_j$ has linear growth at infinity. When $m_j(x)$ grows faster than quadratically as $|x|\to\infty$, it is impossible to use Theorem \ref{Corollary to C11 Whitney thm for coercive convex functions} with functions of this form.}
$$
\varphi_{(j,u,v)}(x,y)=A_{j,u,v}\left( (x-u)^2+ (y-v)^2\right),
$$
where $A_{j, u, v}$ are positive numbers depending only on $j, u, v$.
We have to check that for every $j\in\N$ and $(u,v)\in E$ there exists some number $A=A_{j, u, v}>0$ so that for every $(x,y)\in\R^2$ we have that
\begin{multline}\label{minimal below g in example showing that the trace norms blow up}
\max\left\{|x|, 2(y-j-1)\right\}\leq \\
\begin{cases}
-x+A\left( (x-u)^2+(y-v)^2\right) & \textrm{ if } (u,v)\in E_1, u<0 \\
x +A\left( (x-u)^2+(y-v)^2\right) & \textrm{ if } (u,v)\in E_1, u>0 \\
-x+A\left( (x-u)^2+(y-v)^2\right) & \textrm{ if } (u,v)\in E_2, u<0, 1\leq v\leq j+1 \\
x+A\left( (x-u)^2+(y-v)^2\right) & \textrm{ if } (u,v)\in E_2, u>0, 1\leq v\leq j+1 \\
2(y-j-1)+A\left( (x-u)^2+(y-v)^2\right) & \textrm{ if } (u,v)\in E_2, v>j+1.
\end{cases}
\end{multline}
To this end, let us consider the functions $h_i=h_{i}^{(j,u,v)}:\R^2\to\R$, $i=1, ..., 9$, $j\in\N$, $(u,v)\in E$, defined by
\begin{eqnarray*}
h_1(x,y) = & -x +A\left( (x-u)^2+(y-v)^2\right) -|x|  & \textrm{ if } (u,v)\in E_1, u<0\\
h_2(x,y) = & x +A\left( (x-u)^2+(y-v)^2\right) -|x|  & \textrm{ if } (u,v)\in E_1, u<0\\
h_3(x,y) = & -x +A\left( (x-u)^2+(y-v)^2\right) -2(y-j-1)  & \textrm{ if } (u,v)\in E_1, u<0 \\
h_4(x,y) = & x +A\left( (x-u)^2+(y-v)^2\right) -2(y-j-1)  & \textrm{ if } (u,v)\in E_1, u>0 \\
h_5(x,y) = & -x +A\left( (x-u)^2+(y-v)^2\right) -|x|  &  \textrm{ if } (u,v)\in E_2, u<0, 1\leq v\leq j+1 \\
h_6(x,y) = & x +A\left( (x-u)^2+(y-v)^2\right) -|x|  &  \textrm{ if } (u,v)\in E_2, u>0, 1\leq v\leq j+1 \\
h_7(x,y) = & -x +A\left( (x-u)^2+(y-v)^2\right) -2(y-j-1)  & \textrm{ if } (u,v)\in E_2, u<0, 1\leq v\leq j+1 \\
h_8(x,y) = & x +A\left( (x-u)^2+(y-v)^2\right) -2(y-j-1)  & \textrm{ if } (u,v)\in E_2, u>0, 1\leq v\leq j+1 \\
h_9(x,y) = & 2(y-j-1) +A\left( (x-u)^2+(y-v)^2\right) -|x|  & \textrm{ if } (u,v)\in E_2, v>j+1.
\end{eqnarray*}
For each $j\in\N$, $(u,v)\in E$, we want to find some $A=A_{j,u,v}\geq 0$ such that these functions satisfy $h_i^{(j,u,v)}(x,y)\geq 0$ for all $(x,y)\in\R^2$. Finding the minima of these piecewise quadratic functions is routine. We have
$$
h_2(x,y)\geq 
\begin{cases}
h_2(u, v)=0 & \textrm{ if } x\geq 0 \\
h_2(u-\tfrac{1}{A}, v)=2u-\frac{1}{A} & \textrm{ if } x\leq 0,
\end{cases}
$$
and since in this case we have $(u, v)\in E_1$, $u>0$, we obtain that
$h_2(x,y)\geq 0$ for all $(x,y)\in\R^2$ provided that 
$
A\geq \frac{1}{2u}.
$
Similarly, or just noting that $h_{1}^{(j,u,v)}(x,y)=h_{2}^{(j,-u,v)}(-x,y)$, we also obtain that $h_1(x,y)\geq 0$ if we take $A\geq \frac{1}{2|u|}.$ 

On the other hand, bearing in mind that  $u\geq e^v$ when $(u, v)\in E_1$ and $u>0$, we have
\begin{eqnarray*}
& & h_4(x,y)\geq h_4(u-\tfrac{1}{2A}, v+\tfrac{1}{A}) =
u-2v+2(j+1)-\tfrac{5}{4A}\\
& & \geq e^v -2v +2(j+1)-\tfrac{5}{4A}\geq
2(1-\log 2)+2(j+1)-\tfrac{5}{4A}\\
& &
\geq 2(j+1)-\tfrac{5}{4A}\geq 0
\end{eqnarray*}
provided that we further require that 
$
A\geq \frac{5}{8(j+1)}.
$
Noting that $h_{3}^{(j, u,v)}(x,y)=h_{4}^{(j, -u,v)}(-x,y)$, we also obtain that $h_3(x,y)\geq 0$ for such an $A$. 

Next, for $i=5,6$ we have
$$
h_{5}^{(j,u,v)}(x,y)=h_{6}^{(j,-u,v)}(-x,y),
$$
and also (noticing that $u=1$ when $(u,v)\in E_2, u>0$)
$$
h_6(x,y)\geq \min\left\{0, 2u-\tfrac{1}{A}\right\}=
\min\left\{0, 2-\tfrac{1}{A}\right\}\geq 0,
$$
provided that we take $A\geq \tfrac{1}{2}$. 

For $i=8$, recalling that $(u,v)\in E_2, u>0$ and $v\leq j+1$ if and only if $u=1$, $v\in\N $ and $v\leq j+1$, we get 
\begin{multline*}
h_8(x,y)\geq h_8\left(u-\tfrac{1}{2A}, v+\tfrac{1}{A}\right)=
u-2v+2(j+1)-\tfrac{5}{4A} \\ =1+2(j+1-v)-\tfrac{5}{4A}\geq 1-\tfrac{5}{4A}\geq 0,
\end{multline*}
whenever $A\geq 5/4,$
and since $h_{7}^{(j, u,v)}(x,y)=h_{8}^{(j, -u,v)}(-x,y)$, we also obtain that $h_7(x,y)\geq 0$ with the same $A$.

Lastly, for $i=9$, noting that if $(u,v)\in E_2$ then $v\in\N, v\geq j+2$, $|u|=1$, we have
\begin{multline*}
h_9(x,y)\geq 
\min\{h_9\left(u+\tfrac{1}{2A}, v-\tfrac{1}{A}\right), h_9\left(u-\tfrac{1}{2A}, v-\tfrac{1}{A}\right) \} \\ \geq
2\left(v-j-1\right) -|u|-\tfrac{5}{4A}\geq 1-\tfrac{5}{4A}
\geq 0
\end{multline*} provided that $A\geq 5/4$.

In conclusion we see that inequality \eqref{minimal below g in example showing that the trace norms blow up} is satisfied for 
$$
A=A_{j,u,v}=
\max\left\{\frac{1}{2|u|}, \frac{5}{4} \right\}.
$$
Also note that, for each $R\geq 1$, since $\frac{1}{|u|}\leq e^{R}$ for all $(u, v)\in E\cap B(0,R)$,  we have
\begin{align*}
& & M_{j,R}:=\sup\left\{\textrm{Lip}\left( (\varphi_{(j,u,v)})_{|_{B(0,R)}}\right) : (u,v)\in E\cap B(0,R)\right\} \\ & & =
\sup\left\{ 2A_{j,u,v} : (u,v)\in E\cap B(0,R)\right\} 
\leq \max\left\{e^{R}, \tfrac{10}{4} \right\}<\infty.
\end{align*}
Therefore we can apply Theorem \ref{Corollary to C11 Whitney thm for coercive convex functions} so as to obtain,
for 
each $j\in\N$, a convex function $F_j\in C^{1,1}_{\textrm{loc}}(\R^n)$ such that $(F_j, \nabla F_j)_{|_E}=(f_j, G_j)$. We have thus proved $(1)$.

To prove $(2)$, taking for instance $(x_0,y_0)=(1,1)\in E$ and setting
$$
\mu_{0, E}(f_j, G_j)=|f_j(1,1)|+|G_j(1,1)|=2,
$$
we note that the preceding estimate for $M_{j, R}$ implies that
\begin{equation}
\mu_{k, E}(f_j, G_j)\leq \max\left\{e^{k}, 3 \right\} \textrm{ for all } k, j\in\N\cup\{0\},
\end{equation}
and therefore
\begin{equation}
\sup_{j\in\N}\mu_{k, E}(f_j, G_j)<\infty \textrm{ for all } k\geq 0.
\end{equation}
It is also easy to see that $\sup_{j\in\N}\rho^{W}_{k, E}(f_j, G_j)<\infty$ and  $\sup_{j\in\N}\rho^{CW}_{k, E}(f_j, G_j)<\infty$ for all $k\geq 0$. This shows $(2)$.

Finally, let us prove $(3)$. 
Let $\{H_j\}_{j\in\N}$ be a sequence of convex functions of class $C^{1,1}_{\textrm{loc}}(\R^n)$ such that $(H_j, \nabla H_j)_{|_E}=(f_j, G_j)$ for every $j\in\N$, and assume that we had
$$
\sup_{j\in\N}\textrm{Lip}\left( (\nabla H_j)_{|_{B(0,k)}}\right)<\infty \textrm{ for every } k\geq 1.
$$
Since we also have $|H_j(1,1)|+|\nabla H_j(1,1)|=2$ for every $j$, then, for $k=1$, we can apply Arzel\`a-Ascoli's theorem to find a subsequence $\{H_{1,j}\}$  of $\{H_{j}\}$ such that $\{H_{1,j}\}$ and $\{\nabla H_{1,j}\}$ converge uniformly on $B(0,2)$. Then we can apply again Arzel\`a-Ascoli's theorem to find a subsequence $\{H_{2,j}\}$  of $\{H_{1,j}\}$ such that $\{H_{2,j}\}$ and $\{\nabla H_{2,j}\}$ converge uniformly on $B(0,3)$. Continuing this argument by induction, we extract subsequences $\{H_{k,j}\}_{j\in\N}$ of $\{H_{k-1,j}\}_{j\in\N}$ such that $\{H_{k,j}\}$ and $\{\nabla H_{k,j}\}$ converge uniformly on $B(0,k+1)$. Then the diagonal subsequence $\{H_{j,j}\}$ has the property that $\{H_{j,j}\}$ and $\{\nabla H_{j,j}\}$ converge uniformly on $B(0,k)$ for every $k\geq 1$. We deduce that the limit
$$
\lim_{j\to\infty}H_{j,j}(x,y):=H(x,y),
$$
exists locally uniformly, that $H\in C^{1,1}_{\textrm{loc}}(\R^n)$ and also
$$
\nabla H(x,y)=\lim_{j\to\infty}\nabla H_{j,j}(x,y)
$$
locally uniformly. Moreover, since the pointwise limit of convex functions is convex, we have that $H$ is convex. Also, because $\lim_{j\to\infty}H_j(\pm 1, n)=\lim_{j\to\infty}f_j(\pm 1,n)=1$, we have that $H(\pm 1,n)=1$ for every $n\in\N$. And of course, since $H_j(x,y)=f_j(x,y)$ for all $(x,y)\in E$ we have $H(x,y)=|x|$ if $|x|\geq e^{y}$.

Summing up, we have obtained a convex function $H\in C^{1,1}_{\textrm{loc}}(\R^2)$ such that
$H(x,y)=|x|$ for all $(x,y)\in E$. As we are about to see, this implies that $H(x, y)=|x|$ for all $(x,y)\in\R^2$, and in particular $H$ cannot be differentiable at any point of the line $x=0$, a contradiction. Indeed, for every $(x_0, y_0)\in\R^2$ we have
$$
1=H(1,n)\geq H(x_0, y_0)+(1-x_0)\frac{\partial H}{\partial x}(x_0,y_0)+
(n-y_0)\frac{\partial H}{\partial y}(x_0,y_0)\textrm{ for all } n\in\Z,
$$
which implies 
$$
\frac{\partial H}{\partial y}(x_0,y_0)=0
$$
for all $(x_0,y_0)\in\R^2$.  Then, for each $x\in\R$, the function $\R\ni y\mapsto H(x,y)\in\R$ does not depend on $y$. Since for every $(x,y)\in \R^2$ with $x\neq 0$ there exists some $y_0$ with $(x,y_0)\in E_1$, we deduce that $H(x,y)=H(x,y_0)=|x|$. Thus $H(x,y)=|x|$ for all $(x,y)\in\R^2$ with $x\neq 0$, and by continuity also for all $(x,y)\in\R^2$.

This argument shows that we must have $$\sup_{j\in\N}\rho_{k}(H_j)=
\sup_{j\in\N}\textrm{Lip}\left( (\nabla H_j)_{|_{B(0,k)}}\right)=\infty$$ for some $k=k_0\geq 1$ (hence also for all $k\geq k_0$).
\end{proof}

\medskip

\begin{rem}\label{remark on the local Lip constants of the gradient of F}
{\em As we have just shown, there cannot be any method for $C^{1,1}_{\textrm{loc}}$ convex extension of jets that allows us to control the Fr\'echet seminorms of the extensions in terms of the functionals $\rho^{W}_{k, E}$, or $\rho^{CW}_{k, E}$, or
$\mu_{k, E}$. If one needs to estimate the Lipschitz constant of the restriction of the function $F$ of \eqref{formula for Fa in thm2} to some ball $B(0, k)$, by keeping track of the constants and radii appearing in the proof of Theorem \ref{Corollary to C11 Whitney thm for coercive convex functions}, denoting $\nu(R):=M_R$ (the function given by condition \eqref{finite sup of Lip constants of phiy on balls corollary}), and assuming without loss of generality that $\eta(R)\geq 2R$,
where $\eta(R)$ is given by \eqref{the inf restricts to balls condition}, and
that $k\geq R_0$,
where $R_0=|z_0|$ for some $z_0\in E$,
we see that
\begin{eqnarray*}
& & \rho_{k}(F):=
\sup\left\{\frac{|\nabla F(x)-\nabla F(y)|}{|x-y|} \, : \, x, y\in B(0,k), x\neq y\right\}\leq \\
 & &
(n+1)\nu\left( \eta\left( (n+1)\left( k+\frac{1}{\delta}(n+1)\left( |f(z_0)| +2k|G(z_0)| +2k^2 \nu(k)\right)+\frac{1}{\delta^2}\right)\right)\right),
\end{eqnarray*}
where $\delta>0$ is any number such that for some $v\in\R^n$ the function $x\mapsto m(x)-\langle v, x\rangle$ is coercive (where $m(x):=\sup_{y\in E}\{f(y)+\langle G(y), x-y\rangle\}$) and $m(x)-m(z_0)-\langle v, x-z_0\rangle \geq \delta|x-z_0|-\frac{1}{\delta}$  for all $x\in\R^n$. On the other hand, the proof of Lemma \ref{if varphiy then varphitilde restricts} shows that for Theorem \ref{Corollary to C11 Whitney thm for coercive convex functions} one can take
$$
\eta(R)=R+(R+R_0)\sqrt{1+\nu(R)/2a}.
$$
As we see (even if we take $a=1$) these bounds not only depend on $n$, $k$ and $\nu$, but also on the number $\delta$, which somehow measures essential coerciveness of the function $m(x)$. This kind of dependence is inevitable: unless $g$ satisfies a global estimate of the kind $g(x+h)+g(x-h)-2g(x)\leq C|h|^2$, in order that $F=\textrm{conv}(g)$ be differentiable, the function $g$ must be essentially coercive. The less essentially coercive $g$ is, the greater the estimates of the local Lipschitz constants of the gradient of $F$ are bound to be. On the other hand, in the proof of the preceding proposition we saw that the seminorms of the extensions $H_j$ blow up as the functions $H_j(x,y)$ are forced to be  closer and closer to $|x|$ when $j\to\infty$. This indicates that, for any extension operator
$$
\mathcal{J}^{1,1 \, \textrm{loc}}_{\textrm{conv}}(E)\ni (f, G)\mapsto \mathcal{E}(f,G)\in 
C^{1,1 \, \textrm{loc}}_{\textrm{conv}}(\R^n),
$$
a measure of essential coerciveness of the minimal extension functions $m_{(f,G)}(x)=\sup_{y\in E}\{f(y)+\langle G(y), x-y\rangle\}$ defined by a given family of jets $(f, G)$ is a factor that one must consider if one wishes to be able to control the seminorms $\rho_{k}(\mathcal{E}(f, G))$ of the resulting family of extensions. In this direction, the above estimate for $\rho_k(F)$ yields the following result (for simplicity we only consider the case that $\textrm{span}\{G(y)-G(z): y,z\in E\}=\R^n$).
}
\end{rem}

For a point $x_0\in E$ and a $1$-jet $(f, G)$ on $E$, let us denote, for $k\geq 1$,
$$
\mu_{k, E, x_0}(f,G):= \\
\inf_{\varphi_y}\left\{  \sup\left\{\frac{|\nabla\varphi_{y}(u)-\nabla\varphi_{y}(v)|}{|u-v|} \, : \, y\in E\cap B(x_0, k)), u, v\in B(x_0, k), u\neq v \right\}\right\},
$$
where the infimum is taken over all the families of functions $\{\varphi_y\}$ satisfying the conditions of Theorem \ref{MainTheorem with P}. If there exists no such family, we deem $\mu_{k, E, x_0}(f, G)=\infty$ for all $k$. Define also
$$
\mu_{0, E, x_0}(f, G)=|f(x_0)|+|G(x_0)|.
$$
Similarly, for any function $H\in C^{1,1}_{\textrm{loc}}(\R^n$) and $k\in\N$, let us denote
$$
\rho_{k, x_0}(H)=\textrm{Lip}\left( (\nabla H)_{|_{B(x_0, k)}}\right),
$$
and also
$$
\rho_{0, x_0}(H)=|H(x_0)|+|\nabla H(x_0)|.
$$
\begin{thm}\label{when uniformly essentially coercive it is possible to control seminorms}
Let $(f_{\alpha}, G_{\alpha})_{\alpha\in\mathcal{A}}$ be a family of $1$-jets on a nonempty subset $E$ of $\R^n$. Assume that these jets are {\em uniformly essentially coercive}, in the sense that there exist some $\delta>0$ and some point $x_0\in E$ such that for every $\alpha\in\mathcal{A}$ there exists a vector $v_{\alpha}\in\R^n$ so that
$$
m_{\alpha}(x):=\sup_{y\in E}\{f_{\alpha}(y)+\langle G_{\alpha}(y), x-y\rangle\}\geq 
f_{\alpha}(x_0)+ \langle v_{\alpha}, x-x_0\rangle + \delta |x-x_0|-\tfrac{1}{\delta} 
$$
for all $x\in\R^n$. Assume also that for every $\alpha\in\mathcal{A}$ the jet $(f_{\alpha}, G_{\alpha})$ satisfies the conditions of Theorem \ref{Corollary to C11 Whitney thm for coercive convex functions}, and that
$$
\sup_{\alpha\in\mathcal{A}}\mu_{k, E, x_0}(f_{\alpha}, G_{\alpha})<\infty \textrm{ for every } k\in \N\cup\{0\}.
$$
Then, calling $F_{\alpha}$ the extension of $(f_{\alpha}, G_{\alpha})$ given by formula \eqref{formula for Fa in thm2} with $a=1$, we have that
$$
\sup_{\alpha\in\mathcal{A}}\rho_{k, x_0}(F_{\alpha})<\infty \textrm{ for every } k\in \N\cup\{0\}.
$$
\end{thm}

\section{Some applications}

As we already mentioned, our results are essential in the proof of the following theorem from \cite{AzagraHajlasz}, which tells us that essentially coercive convex functions satisfy a {\em Lusin property of class $C^{1,1}_{\textrm{loc}}$ and convex.}

\begin{thm}[Azagra-Haj\l asz]
Let $f:\R^n\to\R$ be a convex function, and assume that $f$ is not of class $C^{1,1}_{\textrm{loc}}(\R^n)$. Then $f$ is essentially coercive if and only if for every $\varepsilon>0$ there exists a convex function $g:\R^n\to\R$ of class $C^{1,1}_{\textrm{loc}}(\R^n)$ such that $\mathcal{L}^{n}\left(\{x\in \R^n : f(x)\neq g(x)\}\right)\leq\varepsilon$. 
\end{thm}
A corollary of this result is that, if $S$ is the boundary of some convex set with nonempty interior (not necessarily bounded) in $\R^n$ and $S$ does not contain any line, then for every $\varepsilon>0$ there exists a convex hypersurface $S_{\varepsilon}$ of class $C^{1,1}_{\textrm{loc}}$ such that $\mathcal{H}^{n-1}(S\setminus S_{\varepsilon})<\varepsilon$; see \cite[Corollary 1.13]{AzagraHajlasz}.

We next present and prove two other interesting consequences of our main results.

\subsection{Convex hypersurfaces of class $C^{1,1}_{\textrm{loc}}$ with prescribed tangent hyperplanes}

Theorem \ref{First variant of MainTheorem with P} can be applied to solve the following natural geometrical problem: given an arbitrary subset $E$ of $\R^n$ and a collection $\mathcal{H}$ of affine hyperplanes of $\R^n$ such that every $H\in\mathcal{H}$ passes through some point $x_{H}\in E$, and $E=\{x_H : H\in\mathcal{H}\}$,
what conditions on $\mathcal{H}$ are necessary and sufficient for the existence of a {\em convex} hypersurface $S$ of class $C^{1,1}_{\textrm{loc}}$ in $\R^n$ such that $H$ is tangent to $S$ at $x_H$ for every $H\in\mathcal{H}$?\footnote{ We say that $S$ is a convex hypersurface $S$ of class $C^{1,1}_{\textrm{loc}}$ provided that $S=\partial W$ for some convex body (possibly unbounded) and $S$ is a $C^1$ submanifold of $\R^n$ such that the outer unit normal $n_{S}(x), x\in S$, is a locally Lipschitz mapping (equivalently, $S$ can be regarded locally,  in appropriate coordinates, as the graph of a $C^{1,1}_{\textrm{loc}}$ function).}
An equivalent reformulation of this problem is the following: given $C\subset\R^n$ and $N: E \to \mathbb{S}^{n-1}$, what conditions are necessary and sufficient to ensure the existence of a (not necessarily bounded) convex body $W$ of class $C^{1,1}_{\textrm{loc}}$ such that $E \subseteq \partial W$ and the outer unit normal $n_{S}(x)$ to $S:=\partial W$ at $x$ coincides with $N(x)$ for every $x\in E$?\footnote{We say that $W$ is a convex body of class $C^{1, 1}_{\textrm{loc}}$ if $W$ is a closed convex subset of $\R^n$ with nonempty interior such that its boundary $\partial W$ is a $C^{1,1}_{\textrm{loc}}$ hypersurface of $\R^n$.} Our solution to this problem is as follows.

\begin{thm}\label{geometric corollary}
Let $E$ be an arbitrary nonempty subset of $\R^n$, $N:E\to\mathbb{S}^{n-1}$ a locally Lipschitz mapping, $X$ a linear subspace of $\R^n$, and $P:\R^n\to X$ the orthogonal projection. Then there exists a convex hypersurface $S$ of class $C^{1,1}_{\textrm{loc}}$ such that $E\subset S$, $N(x)=n_{S}(x)$ for all $x\in E$, and 
$X=\textrm{span}\{ n_{S}(x)-n_{S}(y) : x, y\in S\}$,
if and only if the following conditions are satisfied. 
\begin{itemize}
\item[$(i)$] $Y:=\textrm{span}\{ N(y)-N(z) : y, z\in E\}\subseteq X$.
\item[$(ii)$] If $\ell:=\dim Y< d:= \dim X$, then there exist points $p_1, \ldots, p_{d-\ell} \in \R^n \setminus E$, vectors $w_1, \ldots, w_{d-\ell} \in \mathbb{S}^{n-1}$, and a sequence of numbers $A_k\geq 2$, $k\in\N$, such that, 
denoting: $E^{*}:=E\cup\{p_1, ..., p_{d-\ell}\}$; $\xi_y:=N(y)$ for $y\in E$; $\xi_y=w_i$ for $y=p_i$,
$i=1, ..., d-\ell$,
we have that
\begin{equation}\label{completing the span of derivatives 4}
\textrm{span}\{\xi_y -\xi_z: y, z\in E^{*}\}=X;
\end{equation}
and
\begin{equation}\label{every tangent function lies above all tangent planes Geometrical Corollary}
\langle \xi_z, x-z\rangle
\leq \langle \xi_y, x-y\rangle +\frac{A_k}{2}|Px-Py|^2 	
\end{equation}
for all $z\in E^{*}$, $y\in E^{*}\cap P^{-1}(B_X(0, k))$, $x\in P^{-1}(B_X(0, 4k))$.
\item[$(iii)$] If $\ell=d$, the preceding condition holds with $E$ in place of $E^{*}$.
\end{itemize}
\end{thm}
Before showing this result, let us gather some facts concerning the geometry of unbounded convex bodies that will help us understand its statement and proof. We say that a convex body is {\em line-free} if it does not contain any line (however, it may contain half-lines).
\begin{lem}\label{Line free decomposition of body}
For every convex body $W\subset\R^n$ there exists a linear subspace $Y$ of $\R^n$ such that 
$$
W=(W\cap Y)\times Y^{\perp},
$$
where the convex body $W\cap Y$ is line-free (and possibly unbounded). Furthermore, $Y^{\perp}$ is the set of vectors parallel to lines contained in $W$. Consequently, if $S:=\partial W$ and we denote $P:\R^n\to Y$ the orthogonal projection, we also have: 
\begin{enumerate}
\item $S=(S\cap Y)\times Y^{\perp}$, and
\item $d(x, S)=d\left( P(x), S\cap Y\right)$ for every $x\in\R^n$.
\end{enumerate}
\end{lem}
\begin{proof}
For the first part, see \cite[Lemma 1.4.2]{Schneider} for instance. Then properties $(1)$ and $(2)$ are immediate consequences of the cylindrical structure of $W$.
\end{proof}
The following result must be known, but I have been unable to find a proof in the literature. 
\begin{prop}\label{span of differences of normals proposition}
Let $W\subset \R^n$ be a convex body such that $S:=\partial W$ is a hypersurface of class $C^{1,1}_{\textrm{loc}}$. 
\begin{enumerate}
\item If $W$ is bounded then $\textrm{span}\{ n_S(x) : x\in S\}= \R^n=\textrm{span}\{n_S(x)-n_S(y) : x, y\in S\}$.
\item If $W$ is a halfspace then, with the notation of Lemma \ref{Line free decomposition of body}, $Y$ is $1$-dimensional, $Y^{\perp}$ is a hyperplane parallel to $S$, and $\textrm{span}\{ n_S(x) : x\in S\}=Y$, but $\textrm{span}\{n_S(x)-n_S(y) : x, y\in S\}=\{0\}$.
\item If $W$ is unbounded and is not a halfspace then, with the notation of Lemma \ref{Line free decomposition of body}, we have
\begin{equation}\label{span of differences of normals equals span of normals}
\textrm{span}\{ n_S(x) : x\in S\}= Y=\textrm{span}\{n_S(x)-n_S(y) : x, y\in S\}
 =
\textrm{span}\{ n_S(x) : x\in S\cap Y\}.
\end{equation}
\end{enumerate}
\end{prop}
\begin{proof}
$(1)$ If $W$ is bounded then, for each $u\in\mathbb{S}^{n-1}$, $\sup_{x\in W}\langle x, u\rangle$ is attained at some $x_u\in \partial W$, and this means that $u=n_S(x_u)$. Hence $\textrm{span}\{ n_S(x) : x\in S\}= \R^n$. Also, for any $y_0\in S$, $\textrm{span}\{ n_S(x)-n_S(y): x, y\in\R^n\}$ contains the sphere of center $-n_S(y_0)$ and radius $1$, and therefore must coincide with $\R^n$.

\noindent $(2)$ is obvious.

\noindent $(3)$ If $W$ is unbounded, according to Lemma \ref{Line free decomposition of body} let us write $W=(W\cap Y)\times Y^{\perp}$, where $W\cap Y$ is line-free (and $Y$ may be equal to $\R^n$). From the cylindrical structure of $W$ we see that 
\begin{equation}\label{normal restricted}
n_{S\cap Y}(P(x))=n_S(x) \textrm{ for every } x\in S,
\end{equation} 
and $$\textrm{span}\{n_S(x): x\in S\}=\textrm{span}\{n_S(x) : x\in S\cap Y\}=
\textrm{span}\{n_{S\cap Y}(x) : x\in S\cap Y\}\subseteq Y.$$
Let us now distinguish some cases depending on the dimension of $Y$.
If $\textrm{dim} Y=0$ then $W=\R^n$, $S=\emptyset$, and there is nothing to say. If $\textrm{dim} Y=1$ then $W$ is either a halfspace (a case already dealt with) or a {\em slab} (the intersection of two parallel halfspaces perpendicular to $Y$ and facing opposite directions). In the latter case it is clear that \eqref{span of differences of normals equals span of normals} is true. 

So we are left with the case that $\textrm{dim}Y\geq 2$. In this case $S$ is connected, and if $\textrm{span}\{n_S(x) : x\in S\cap Y\}$ were strictly contained in $Y$ then $W\cap Y$ would be contained in a proper subspace of $Y$ and therefore $W=(W\cap Y)\times Y^{\perp}$ would have empty interior, which is absurd. Thus we have
\begin{equation}\label{normals to cylinder}
\textrm{span}\{n_S(x): x\in S\}=\textrm{span}\{n_S(x) : x\in S\cap Y\}=\textrm{span}\{n_{S\cap Y}(x) : x\in S\cap Y\}= Y.
\end{equation}
For notational convenience, let us first assume that $W$ is line-free, that is, $\R^n=Y=\textrm{span}\{ n_S(x) : x\in S\}$, and check that $\textrm{span}\{n_S(x)-n_S(y) : x, y\in S\}=\R^n$ too. Let us choose points $x_1, ..., x_n\in S$ such that $\{n_S(x_1), ..., n_S(x_n)\}$ is a basis of $\R^n$.
\begin{claim}
The set $\Lambda:=\{\sum_{j=1}^{n} \lambda_j n(x_j) :\lambda_j>0, j=1, ..., n\}\cap\mathbb{S}^{n-1}$ is contained in $\{n_S(x): x\in S\}$.
\end{claim}
\begin{proof}
Let us denote $H_{j}^{-}=\{x\in\R^n : \langle x, n_S(x_j)\rangle \leq \langle x_j, n_S(x_j)\rangle\}$ and $H_j=\{x\in\R^n : \langle x, n_S(x_j)\rangle = \langle x_j, n_S(x_j)\rangle\}=\partial H_{j}^{-1}$ for $j=1, ..., n$. We have $W\subseteq C:=\bigcap_{j=1}^{n}H_{j}^{-}$, and since $\{n_S(x_1), ..., n_S(x_n)\}$ is a basis of $\R^n$ the hyperplanes $H_j$, $j=1, ..., n$, must intersect at a unique point $p_0$, which is the vertex of the pointed cone $C$. Given $\mu_1, ..., \mu_n>0$, we set $v :=\sum_{j=1}^{n}\mu_j n_S(x_j)$, $u:=v/|v|$, and observe that the hyperplanes $\{x: \langle x, u\rangle =r\}$ intersect $C$ transversely at least for all $r\leq \langle x_1, u\rangle$. Then, for $r\leq \langle x_1, u\rangle$, the truncated cone $C_{u, r}:=\{x\in C : r\leq \langle x, u\rangle\}$ is nonempty and compact, hence so is $K_{u, r}:=W\cap C_{u, r}$, and therefore $\sup_{x\in K_{u,r}}\langle x, u\rangle$ is attained at some $x_{u,r}\in K_{u,r}$ But since $W\subseteq C$ and $\langle x, u\rangle<r\leq \langle x_1, u\rangle$ for all $x\in W\setminus C_{u, r}$, we have that $\sup_{x\in K_{u,r}}\langle x, u\rangle=\sup_{x\in W}\langle x, u\rangle$ is attained at $x_{u,r}$, and this implies that $x_{u, r}\in\partial W=S$ and $n_{S}(x_{u, r})=u$.
\end{proof}
Now, since $\Lambda$ is open in the unit sphere $\mathbb{S}^{n-1}$, for any $y_0\in S$ we have that $-n(y_0)+\Lambda$ is open in the sphere of center $-n(y_0)$ and radius $1$, and (because any nonempty relatively open subset of a sphere spans all of $\R^n$) it follows that $\textrm{span}\{ n_S(x)-n_S(y_0) : x\in S\}=\R^n$ , which yields $\textrm{span}\{ n_S(x)-n_S(y) : x, y\in S\}=\R^n$.

Let us finally consider the case that $Y\neq\R^n$. By applying what we have just established to the convex body $W\cap Y$ (with boundary $S\cap Y)$ of the space $Y$, we see that $\textrm{span}\{ n_{S\cap Y}(x) : x\in S\cap Y\}= Y=\textrm{span}\{n_{S\cap Y}(x)-n_{S\cap Y}(y) : x, y\in S\cap Y\}$, and by combining this with \eqref{normal restricted} and \eqref{normals to cylinder} we conclude the proof of $(3)$.
\end{proof}

\begin{proof}[{\bf Proof of Theorem \ref{geometric corollary}}]
Let us assume that conditions $(i)-(iii)$ are satisfied and, with the help of Theorem \ref{First variant of MainTheorem with P}, let us construct a convex hypersurface $S$ as required. Define $f$ and $G$ on $E^{*}$ by $f(y)=0$ and $G(y)=\xi_y$. Then\eqref{completing the span of derivatives 4} implies \eqref{completing the span of derivatives 2}, and \eqref{every tangent function lies above all tangent planes Geometrical Corollary} implies \eqref{every tangent function lies above all tangent planes MainthmwithP 2}, so we can apply Theorem \ref{First variant of MainTheorem with P} to obtain a convex function $F\in C^{1,1}_{\textrm{loc}}(\R^n)$ such that $(F, \nabla F)=(f, G)$ on $E^{*}$ and $\textrm{span}\{ \nabla F(x)-\nabla F(y) : x, y\in \R^n\}=X$. Note that $F$ is not constant because $\nabla F(y)=\xi_y\neq 0$ for any $y\in E$, where we have $F(y)=0$. Since a convex function has vanishing gradients exactly at the points where a global minimum is attained, it is clear that for every $x\in F^{-1}(0)$ we have $\nabla F(x)\neq 0$. Therefore
$$
W:=F^{-1}(-\infty, 0]
$$ 
defines a convex body of class $C^{1,1}_{\textrm{loc}}$, and its boundary $$S:=\partial W=F^{-1}(0)$$ is a convex hypersurface of class $C^{1,1}_{\textrm{loc}}$. It is obvious that $E\subseteq S$, and since $\nabla F(x)$ points outside $W$ and is perpendicular to $S$ at $x$ for every $x\in S$, and $\nabla F(y)=\xi_y$ for all $y\in E^{*}$, we have that $N=n_{S}$ on $E$ and $\textrm{span}\{ n_{S}(x)-n_{S}(y) : x, y\in S\}=\textrm{span}\{\nabla F(x)-\nabla F(y) : x, y\in \R^n\}=X$.

Conversely, let us assume that there is a convex $C^{1,1}_{\textrm{loc}}$ hypersurface $S=\partial W$ with $X=\textrm{span}\{ n_{S}(x)-n_{S}(y) : x, y\in S\}$, $E\subset S$ and $n_S=N$ on $E$, and let us see that conditions $(i)-(iii)$ of the statement are met. According to Lemma \ref{Line free decomposition of body} and Proposition \ref{span of differences of normals proposition}, we may write $W=(W\cap Z)\times Z^{\perp}$, where $W\cap Z$ is line-free, and we have that $X=Z$ unless $W$ is a halfspace. If $W$ is a halfspace then $X=\{0\}=Y$, $N(y)=N(z)$ for all $y, z\in E$, and conditions $(i)-(iii)$ of the statement are trivially satisfied. Thus we may assume $X=Z$ (and in particular the $P$'s in the statements of Theorem \ref{geometric corollary} and Lemma \ref{Line free decomposition of body} coincide).

For a convex body $V$, let $\varphi_V:\R^n\to\R$ denote the {\em signed distance} to $\partial V$, that is, 
$$
\varphi_V(x)=
\begin{cases}
d(x, \partial V) & \textrm{ if } x\notin V \\
-d(x, \partial V) & \textrm{ if } x\in V.
\end{cases}
$$
By Lemma \ref{Line free decomposition of body} we have $\varphi_W(x)=\varphi_{W\cap X}(P(x))$ for all $x\in\R^n$.
It is well known that if $\partial V$ is a convex hypersurface of class $C^{1,1}_{\textrm{loc}}$ then the function $\varphi_V$ is convex and there exists an open neighborhood $\Omega$ of $\partial V$ such that ${\varphi_{V}}_{|_{\Omega}}$ is of class $C^{1,1}_{\textrm{loc}}(\Omega)$, and $\nabla \varphi_V(x)=n_{\partial V}(x)$ for every $x\in \partial V$; see \cite[Theorems 5.4 and 5.7]{DelfourZolesio}. By applying this result to $V=W\cap X$, we obtain an open neighborhood $U_0$ of $S\cap X$ in $X$ such that $\varphi_{W\cap X}\in C^{1,1}_{\textrm{loc}}(U_0)$, and hence $\varphi_{W}\in C^{1,1}_{\textrm{loc}}(U)$, where $U:=P^{-1}(U_0)$. Now, for every $k\in\N$, since $S\cap X\cap B(0, 4k)$ is compact, there exists numbers $L_k>0$ and $\delta_k\in (0, 1]$ such that $B_X (x, \delta_k)\subseteq U_0$ for all $x\in S\cap X\cap B(0, 4k)$, and $|\nabla\varphi_{W\cap X}(x)-\nabla\varphi_{W\cap X}(y)|\leq L_k|x-y|$ for all $y\in S\cap X\cap B(0, 4k)$ and $x\in B_X(y, \delta_k)$, which (bearing in mind that $\varphi_{W\cap X}^{-1}(0)=S\cap X$, $\nabla\varphi_{S\cap X}=n_{S\cap X}$ on $S\cap X$, and $\varphi_{W\cap X}$ is convex) implies that
$$
\langle n_{S\cap X}(z), x-z\rangle\leq \varphi_{S\cap X}(x) \leq\langle n_{S\cap X}(y), x-y\rangle +L_k |x-y|^{2}
$$
for all $z\in S\cap X$, $y\in S\cap X\cap B(0, 4k)$ and $x\in B_X(y, \delta_k)$. On the other hand, if $y\in S\cap X\cap B(0, 4k)$ and $x\in B_{X}(0, 4k)\setminus  B_{X}(y, \delta_k)$ then, by setting 
$$
A_k:= \frac{2}{{\delta_{k}}^{2}}\left( 8k+L_k +\sup_{x\in B_{X}(0, 4k)}|\varphi_{S\cap X}(x)|\}\right)
$$
we have that
$$
\langle n_{S\cap X}(z), x-z\rangle\leq \varphi_{S\cap X}(x) \leq\langle n_{S\cap X}(y), x-y\rangle +\frac{A_k}{2} |x-y|^{2}.
$$
Thus, in either case, the above inequality holds for every $z\in S\cap X$, $y\in S\cap X\cap B(0, 4k)$, $x\in B_{X}(0,4k)$, and since $\varphi_W(x)=\varphi_{W\cap X}(P(x))$ and $n_S(x)=n_{S\cap X}(P(x))$, we deduce that
\begin{equation}\label{every tangent function lies above all tangent planes normal version proof of Geometrical Corollary}
\langle n_{S}(z), x-z\rangle \leq \langle n_{S}(y), x-y\rangle +\frac{A_k}{2} |P(x-y)|^{2}
\end{equation}
for every $z\in S$, $y\in S\cap P^{-1}(B_X(0, 4k))$, $x\in P^{-1}(B_{X}(0,4k))$. Clearly $(i)$ is always satisfied as $E\subset S$, and if $\ell:=\textrm{dim} Y =d:=\textrm{dim} X$ we are done. 

If $\ell<d$ then  $Y=\textrm{span}\{n_S(x)-n_S(y) : x, y\in E\}$ is strictly contained in $X$, and we can find points $x_0,x_1, \ldots, x_\ell \in E$ such that $Y =\textrm{span} \lbrace n_S(x_j)-n_S(x_0) \: : \: j=1, \ldots, \ell \rbrace.$ Then, by mimicking the beginning of the proof of (ii) in the necessity part of Theorem \ref{MainTheorem with P} below, we may obtain points $p_1, \ldots ,p_{d-\ell} \in \R^n$ such that the set $ \lbrace n_S(p_j)- n_S(x_0) \rbrace_{j=1}^{d-\ell}$ is linearly independent and $X = Y \oplus \textrm{span}\lbrace n_S(p_j)- n_S(x_0) \: : \: j=1, \ldots ,d-\ell \rbrace$, hence
$
X=\textrm{span}\left\{u-w : u, w\in n_S(E^{*}) \right\},
$
where $E^{*}:=E\cup\{p_1, ..., p_{d-\ell}\}$. Thus, if we set $\xi_y:=N(y)$ for $y\in E$, and $\xi_y:=w_i := n_S(y)$ for $y=p_i$,
$i=1, ..., d-\ell$, we see that \eqref{completing the span of derivatives 4} is true, and from \eqref{every tangent function lies above all tangent planes normal version proof of Geometrical Corollary}
we conclude that \eqref{every tangent function lies above all tangent planes Geometrical Corollary} is also satisfied.
\end{proof}

\begin{rem}
{\em By using first the necessity part and then the proof of the sufficiency part of Theorem \ref{geometric corollary} with $E=S$, we deduce the not entirely obvious fact that for every convex hypersurface $S$ of class $C^{1,1}_{\textrm{loc}}$ in $\R^n$ (defined as in Footnote 3) there always exists a convex function $\varphi\in C^{1,1}_{\textrm{loc}}(\R^n)$ such that $\varphi^{-1}(0)=S$ and $\nabla\varphi(x)=n_S(x)$ for every $x\in S$.}
\end{rem}

\subsection{A new formula for (not necessarily convex) $C^{1,1}_{\textrm{loc}}$ extensions of $1$-jets}

A function $f:\R^n\to\R$ is of class $C^{1, 1}_{\textrm{loc}}$ if and only if there exists a coercive convex function $\psi:\R^n\to\R$ of class $C^{1,1}_{\textrm{loc}}$ such that the functions $f+\psi$ and $\psi-f$ are convex and coercive. As we did in \cite{AzagraLeGruyerMudarra} in the $C^{1,1}$ case, one can use this fact in combination with Theorem \ref{Corollary to C11 Whitney thm for coercive convex functions} to obtain explicit formulas for general (not necessarily convex) $C^{1,1}_{\textrm{loc}}$ extensions of jets.

More precisely, if we are given a $1$-jet $(f, G)$ on a set $E\subset\R^n$ and we can guess that for some convex function $\psi\in C^{1,1}_{\textrm{loc}}(\R^n)$ the jet $(f+\psi, G+\nabla\psi)$ will have a coercive $C^{1,1}_{\textrm{loc}}$ convex extension $\widetilde{F}$, then the $C^{1,1}_{\textrm{loc}}$ function $F=\widetilde{F}-\psi$ will extend the original jet $(f, G)$. Thus Theorem \ref{Corollary to C11 Whitney thm for coercive convex functions} for the case $X=\R^n$ has the following consequence.\footnote{Here we make the mild assumption that the set $E$ has at least one subset consisting of $n+1$ affinely independent points, so that we do not have to add new data in some special cases (at least if we choose an appropriate function $\psi$). Of course, a fully general, but also more complicated version of  Theorem \ref{Extension formula for not necessarily convex jets} follows from Theorem \ref{MainTheorem with P} too. We leave its statement to the reader's care.}

\begin{thm}\label{Extension formula for not necessarily convex jets}
Let $E\subset\R^n$ be such that there are points $x_0, x_1, ..., x_n\in E$ so that $\{x_1-x_0, ..., x_n-x_0\}$ is a basis of $\R^n$. Let $f:E\to\R$, $G:E\to\R^n$ be arbitrary functions. Then there exists a function $F\in C^{1,1}_{\textrm{loc}}(\R^n)$ such that
 $F_{|_E}=f$, $(\nabla F)_{|_E}=G$ if and only if there exist a convex function $\psi\in C^{1,1}_{\textrm{loc}}(\R^n)$ and, for $y\in E$, functions $\varphi_{y}:\R^n\to [0, \infty)$ of class $C^{1,1}_{\textrm{loc}}$ such that:
\begin{equation}\label{span of derivatives in general extension formula}
\textrm{span}\{G(y)+\nabla\psi(y)-G(z)-\nabla\psi(z) : y, z\in E\}=\R^n;
\end{equation}
\begin{equation}\label{phiy equals 0 at y for extension of not necessarily convex jets}
\varphi_{y}(y)=0, \nabla\varphi_{y}(y)=0;
\end{equation}
\begin{equation}\label{finite sup of Lip constants of phiy for extension of not necessarily convex jets}
\sup\left\{\frac{|\nabla\varphi_{y}(x)-\nabla\varphi_{y}(z)|}{|x-z|} \, : x, z\in B(0, R), x\neq z,  \, y\in E\cap B(0,R)\right\} <\infty
\end{equation}
for every $R>0$, and
\begin{equation}\label{every tangent function lies ... for extension of not necessarily convex jets}
f(z)+\psi(z)+ \langle G(z)+\nabla\psi(z), x-z\rangle
\leq
 f(y)+\psi(y)+\langle G(y)+\nabla\psi(y), x-y\rangle +\varphi_{y}(x)
\end{equation}
for every $y, z\in E$ and every $x\in \R^n$. 

Moreover, whenever these conditions are satisfied, for every number $a>0$ the formula
\begin{multline}\label{general extension formula}
F(x)= \\
\textrm{conv}\left(x\mapsto \inf_{y\in E}\left\{f(y)+\psi(y)+\langle G(y) +\nabla\psi(y), x-y\rangle+\varphi_{y}(x)+a|x-y|^2\right\}\right)-\psi(x)
\end{multline}
defines a $C^{1,1}_{\textrm{loc}}$ convex extension of the jet $(f, G)$ to $\R^n$.
\end{thm}

\begin{rem}
{\em Once again, in contrast to the $C^{1,1}$ case which we studied in \cite{AzagraLeGruyerMudarra}, the gradient of the function $F$ given by \eqref{general extension formula} does not have optimal local Lipschitz constants. As observed in Remark \ref{remark on the local Lip constants of the gradient of F} and Theorem \ref{when uniformly essentially coercive it is possible to control seminorms}, our method does not provide extensions whose gradients have local Lipschitz constants independent of the dimension or smaller than those given by the classical Whitney operator. Hence we do not recommend using the above formula if the magnitude of the local Lipschitz constants of the gradient is a concern and convexity is not. Nonetheless, its form and its explicit character may become useful in other situations, for instance when dealing with delta-convex functions.}
\end{rem}

\begin{proof}[Proof of Theorem \ref{Extension formula for not necessarily convex jets}]
Assume that the jet $(f, G)$ has a $C^{1,1}_{\textrm{loc}}$ extension $F$. Set $B_0=\emptyset$ and for each $k\in\N$ denote $B_{k}=B(0,k)$ and $M_k=\textrm{Lip}\left(\nabla F_{|_{B_k}}\right)$. Then $F+\frac{1}{2}M_k|\cdot|^2$ and $\frac{1}{2}M_k|\cdot|^2-F$ are convex functions on $B_k$, for each $k\in\N$. Define $\psi_0=0$, and, for $k\geq 1$, 
$$
\psi_k(x)=
\begin{cases}
0 & \textrm{ if } x\in B_{k-1} \\
(1+M_{8k}) \left( |x| -(k-1) \right)^2  & \textrm{ if } x\in\R^n\setminus B_{k-1},
\end{cases}
$$
and
$$
\psi(x)=\sum_{k=1}^{\infty}\psi_{k}(x).
$$
It is clear that the functions $\psi_k, \psi:\R^n\to\R$ are convex and of class $C^{1,1}_{\textrm{loc}}$. Next we check that $F+\psi$ is convex (in fact strongly convex) on $\R^n$.
We can write, on each $B_{4(k+1)}\setminus B_{4k}$,
$$
F+\psi=\left(F+\frac{1}{2}M_{4(k+1)}|\cdot|^{2}\right)+ \left(\psi-\frac{1}{2}M_{4(k+1)}|\cdot|^{2}\right),
$$ 
with $F+\frac{1}{2}M_{4(k+1)}|\cdot|^2$ convex on $B_{4(k+1)}$, and of course $\R^n=\bigcup_{k=0}^{\infty}\left( B_{4(k+1)}\setminus B_{4k}\right)$. Therefore, recalling that $F, \psi\in C^{1,1}_{\textrm{loc}}$, in order to check that $F+\psi$ is strongly convex on $\R^n$ it is sufficient to see that if $x, v\in\R^n$ and $|v|=1$, the second derivative of the function $t\mapsto \beta(t):=\psi(x+tv)- \frac{1}{2}M_{4(k+1)}|x+tv|^2$ (which exists for almost every $t\in\R$) is bounded below by some strictly positive number. In fact this function is twice differentiable on $\R$ except on the countable set $\{t: |x+tv|\in\N\}$. If $t_0$ is a point of differentiability of $\beta'(t)$ and $x+t_0 v\in B_{4(k+1)}\setminus B_{4k}$ then, by calculating the second derivatives at $t=0$ of the convex functions $t\mapsto \alpha_{k}(t):=\psi_k(x+tv)$, one can check that, for $x+t_0v\in B_{4(k+1)}\setminus B_{4k}$ and $|v|=1$ one has
$$
\alpha_{k}''(t_0) \geq (1+M_{8k})\left( 2-\frac{2(k-1)}{|x+t_0 v|}\right)\geq (1+M_{8k}).
$$
and therefore, denoting $\alpha(t)=\psi(x+tv)$, 
$$
\alpha''(t_0)\geq 1+M_{8k}\geq 1 + M_{{4(k+1)}},
$$
hence
$$
\beta''(t_0)\geq 1.
$$
We have seen that $\beta''(t)\geq 1$ for almost every $t\in\R$, and as we noted above this implies that $F+\psi$ is strongly convex on $\R^n$.

If $Y:=\textrm{span}\{\nabla F(y)+\nabla\psi(y)-\nabla F(z)-\nabla\psi(z) : y, z\in E\}=\R^n$ then by applying the necessity part of Theorem \ref{Corollary to C11 Whitney thm for coercive convex functions} to the jet $(\widetilde{f}, \widetilde{G}):=(f+\psi, G+\nabla\psi)$ we immediately get a family of functions $\{\psi_y\}_{y\in E}$ satisfying \eqref{span of derivatives in general extension formula}--\eqref{every tangent function lies ... for extension of not necessarily convex jets}.
Otherwise we proceed as follows. Note that the gradient of the function $\psi$ is of the form 
\begin{equation}\label{gradient of psi at x proportional to x}
\nabla\psi(x)=\lambda(x) x,
\end{equation}
where $\lambda:\R^n\to [0, \infty)$, and $\lambda(x)=0$ if and only if $x=0$. By assumption, there are points $x_0, x_1, ..., x_n\in E$ such that $\{x_1-x_0, ..., x_n-x_0\}$ are linearly independent. Up to replacing the balls $B(0,k)$ with balls $B(x_0, k)$ in the above construction and translating coordinates, we may assume without loss of generality that $x_0=0$ and therefore $\{x_1, ..., x_n\}$ is a basis of $\R^n$. Now, for each $R>1$, consider the function
$$
\psi_R(x)=\psi(Rx),
$$
which clearly has the property that $F+\psi_R$ is strongly convex. We claim that, for $R>1$ large enough, we have
$$
\textrm{span}\{\nabla F(y)+\nabla\psi_R(y)-\nabla F(z)-\nabla\psi_R(z) : y, z\in E\}=\R^n.
$$
Indeed, we have
$
\nabla\psi_R(x)=R\nabla\psi(Rx),
$
so by using \eqref{gradient of psi at x proportional to x} we can write
$$
\nabla \psi_R(x_j)=R^2\lambda_j x_j, \,\,\, j=1, ..., n,
$$
with $\lambda_j>0$, for every $j=1, ..., n$, $R>1$. Then
$$
\frac{1}{R^2}\left(\nabla F(x_j)+\nabla\psi_R(x_j)-\nabla F(0)\right)=\frac{1}{R^2}\left(\nabla F(x_j)-\nabla F(0)\right) +\lambda_j x_j, \,\,\, j=1, ...n,
$$
and by taking the determinants of the matrices formed by the vectors of each side of this equality and letting $R\to\infty$ we obtain
\begin{eqnarray*}
& & \lim_{R\to\infty}\textrm{det}\left(\frac{1}{R^2}\left(\nabla (F+\psi_R)(x_j)-\nabla F(0)\right)\right)_{j=1}^{n} \\
=& &
\lim_{R\to\infty}\textrm{det}\left(\frac{1}{R^2}\left(\nabla F(x_j)-\nabla F(0)\right) +\lambda_j x_j\right)_{j=1}^{n}=
\textrm{det}\left(\lambda_j x_j\right)_{j=1}^{n}\neq 0.
\end{eqnarray*}
Therefore we can find and fix some $R>1$ large enough so that
$$
\textrm{det}\left(\frac{1}{R^2}\left(\nabla (F+\psi_R)(x_j)-\nabla F(0)\right)\right)_{j=1}^{n}\neq 0,
$$
hence also
$$
\textrm{det}\left(\nabla (F+\psi_R)(x_j)-\nabla F(0)\right)_{j=1}^{n}\neq 0,
$$
which since $\nabla\psi_R(0)=0$ shows our claim. Therefore, by applying the necessity part of Theorem \ref{Corollary to C11 Whitney thm for coercive convex functions} to the jet $(\widetilde{f}, \widetilde{G}):=(f+\psi_R, G+\nabla \psi_R)$ we may conclude as before.

Conversely, if there exist a function $\psi$ and functions $\varphi_y$ as in the statement, then by applying Theorem \ref{Corollary to C11 Whitney thm for coercive convex functions} to the jet $(\widetilde{f}, \widetilde{G}):=(f+\psi, G+\nabla\psi)$, we obtain an essentially coercive $C^{1,1}_{\textrm{loc}}$ convex function $\widetilde{F}$ which extends this jet to $\R^n$. Then the $C^{1,1}_{\textrm{loc}}$ function $F:=\widetilde{F}-\psi$ extends the jet $(f, G)$, and the formula for $\widetilde{F}$ given by Theorem \ref{Corollary to C11 Whitney thm for coercive convex functions} yields the formula \eqref{general extension formula} for $F$.
\end{proof}


\section{Proofs of the main results}

Of course Theorem \ref{MainTheorem with P} is more general than Theorem \ref{Corollary to C11 Whitney thm for coercive convex functions}, but its proof is necessarily much more technical and less clear. For this reason, and because Theorem \ref{Corollary to C11 Whitney thm for coercive convex functions} and its consequence Theorem \ref{Corollary 2 to the main coercive result} are powerful enough to have some interesting applications (see, e.g.\,\cite[Theorem 1.12]{AzagraHajlasz}), we choose to prove them separately. 

\subsection{Proof of Theorem \ref{Corollary to C11 Whitney thm for coercive convex functions}, sufficiency.}
The overall strategy is similar to that of the proofs of the main results of \cite{AzagraMudarra1, AzagraLeGruyerMudarra}, and consists in showing that the function
\begin{equation}\label{definition of g}
g(x):=\inf_{y\in E}\left\{f(y)+\langle G(y), x-y\rangle+\varphi_{y}(x)\right\}
\end{equation}
is greater than or equal than the minimal extension
\begin{equation}\label{definition of m}
m(x) := \sup_{z \in E} \lbrace f(z)+ \langle G(z), x-z \rangle \rbrace
\end{equation}
and satisfies estimates of the type $g(x+h)+g(x-h)-2g(x)\leq C_{R}|h|^2$ on each ball $B(0, R)$, and then show that these estimates are preserved, up to some constants, depending on $R$, $n$ and the function $m(x)$, when we take the convex envelope of $g$. 

Observe that \eqref{every tangent function lies above all tangent planes corollary} implies that $m$ and $g$ are finite everywhere; indeed, taking two points $y_0, z_0\in E$, we have

\begin{equation}\label{inequalities that imply boundedness of f and G on bounded sets}
-\infty<f(z_0)+\langle G(z_0), x-z_0\rangle \leq m(x)\leq g(x)\leq f(y_0) +\langle G(y_0), x-y_0\rangle +\varphi_{y_0}(x) <\infty
\end{equation}
for every $x\in\R^n$.
In particular we have
\begin{equation}\label{m is less than g}
m(x)\leq g(x) \textrm{ for all } x\in\R^n.
\end{equation}
Besides $m$ is obviously convex on $\R^n$, and by using conditions \eqref{every tangent function lies above all tangent planes corollary} and \eqref{phiy equals 0 at y corollary} it is easy to see that $m$ is really an extension of $f$, that is, $f(x)=m(x)$ for every $x\in E$. Since convex functions on $\R^n$ are bounded on bounded sets, we see in particular that $f$ is bounded on bounded sets. Using this fact together with \eqref{inequalities that imply boundedness of f and G on bounded sets}, we also deduce that $G$ is bounded on bounded sets.

According to Theorem \ref{rigid global behavior of convex functions}, condition \eqref{essentially coercive data corollary} implies that $m$ is essentially coercive, that is, there exist a convex function $c:\R^n\to\R$ and a vector $v\in\R^n$ such that
$$
m(x)=c(x)+\langle v, x\rangle \textrm{ for all } x\in\R^n,
$$
with $\lim_{|x|\to\infty}c(x)=\infty$. In particular the function $c$ attains a global minimum at some point $x_0\in\R^n$. Hence, up to replacing the jet $(f,G)$ with the jet
$(\widetilde{f}, \widetilde{G})$ defined by $\widetilde{f}(y)=f(y)-c(x_0)-\langle v, y\rangle$, $\widetilde{G}(y)=G(y)-v$, and the function $m(x)$ with $c(x)-c(x_0)$, we may and do assume in the rest of the proof that 
\begin{equation}\label{m is assumed coercive}
\lim_{|x|\to\infty} m(x)=\infty, \,\,\, \textrm{ and } m(x)\geq 0 \textrm{ for all } x\in\R^n
\end{equation}
(note that any function that does not depend on $y$ can be taken in and out of a sum in the infimum defining $g$, and the same goes for any affine function and the convex envelope).

From the definitions of $g$ and $m$, and bearing in mind that $\varphi_{y}(y)=0$ for each $y\in E$, we also obtain $$f(x) \leq m(x) \leq g(x) \leq f(x) \textrm{ for every } x\in E,$$ hence 
\begin{equation}\label{g=f on E}
g(x) = m(x) = f(x) \textrm{ for all } x\in E.
\end{equation}

\begin{lem}\label{if varphiy then varphitilde restricts}
For any number $a>0$, if $\{\varphi\}_{y\in E}$ satisfies conditions \eqref{phiy equals 0 at y}-\eqref{every tangent function lies above all tangent planes corollary} then the family $\{\widetilde{\varphi}_y\}_{y\in E}$ defined by
$$
\widetilde{\varphi}_y(x)=\varphi_y(x)+ a|x-y|^2
$$
satisfies conditions 
\eqref{phiy equals 0 at y}-\eqref{every tangent function lies above all tangent planes corollary} of Theorem \ref{Corollary to C11 Whitney thm for coercive convex functions} (with slightly larger constants in \eqref{finite sup of Lip constants of phiy on balls corollary}), as well as the following one: for every $R>0$ there exists $\eta>R$ such that, for every $x\in B(0, R)$,
\begin{equation}\label{the inf restricts to balls condition}
\inf_{y\in E\cap B(0, \eta)}\{f(y)+\langle G(y), x-y\rangle +\widetilde{\varphi}_{y}(x)\}=
\inf_{y\in E}\{f(z)+\langle G(y), x-y\rangle +\widetilde{\varphi}_{y}(x)\}.
\end{equation} 
\end{lem}
\begin{proof}
It is clear that these new functions $\widetilde{\varphi}_y$ also fulfill conditions 
\eqref{phiy equals 0 at y}, \eqref{finite sup of Lip constants of phiy on balls corollary} and \eqref{every tangent function lies above all tangent planes corollary} of Theorem \ref{Corollary to C11 Whitney thm for coercive convex functions}, with slightly larger constants $$\widetilde{M}_R=M_R+2a$$ in \eqref{finite sup of Lip constants of phiy on balls corollary}. Let us see that the $\widetilde{\varphi}_y$ also satisfy condition \eqref{the inf restricts to balls condition}.
Take $R_0>0$ so that $E\cap B(0, R_0)$ is nonempty, fix a point $y_0\in E\cap B(0, R_0)$, and for any given $R\geq R_0$ note that condition \eqref{finite sup of Lip constants of phiy on balls corollary} implies that
\begin{equation}\label{estimate for phiy0} 
\widetilde{\varphi}_{y_0}(x) \leq \frac{(R+R_0)^2 \widetilde{M}_{R}}{2}=\frac{(R+R_0)^2 (M_R +2a)}{2} \textrm{ for every } x\in B(0, R).
\end{equation}
We then set 
\begin{equation}\label{defn of eta 1}
\eta=\eta(R):= R+ (R+R_0)\sqrt{\widetilde{M_R}/2a}= R+ (R+R_0)\sqrt{1+M_R/2a}.
\end{equation} 
We obtain, for every $y\in E\setminus B(0, \eta)$ and every $x\in B(0,R)$, that
\begin{eqnarray*}
& & f(y)+\langle G(y), x-y\rangle + \widetilde{\varphi}_y(x)=
f(y)+\langle G(y), x-y\rangle + \varphi_y(x)+a|x-y|^2=\\
& & \geq f(y_0)+\langle G(y_0), x-y_0\rangle +a |x-y|^2
\geq f(y_0)+\langle G(y_0), x-y_0\rangle + a\left(\eta-R\right)^2 \\
& &\geq f(y_0)+\langle G(y_0), x-y_0\rangle + \frac{(R+R_0)^2 \widetilde{M}_{R}}{2}
\geq f(y_0)+\langle G(y_0), x-y_0\rangle +\widetilde{\varphi}_{y_0}(x)\\
& &\geq \inf_{z\in E\cap B(0, \eta)}\{f(z)+\langle G(z), x-z\rangle +\widetilde{\varphi}_{z}(x)\}.
\end{eqnarray*}
This shows that
$
\inf_{z\in E\cap B(0, \eta)}\{f(z)+\langle G(z), x-z\rangle +\widetilde{\varphi}_{z}(x)\}=
\inf_{z\in E}\{f(z)+\langle G(z), x-z\rangle +\widetilde{\varphi}_{z}(x)\}.
$
\end{proof}
Hence, up to replacing $\{\varphi_{y}\}_{y\in E}$ with $\{ \widetilde{\varphi}_{y}\}_{y\in E}$, from now on we may and do assume that the family $\{\varphi_{y}\}_{y\in E}$ satisfies conditions \eqref{phiy equals 0 at y}-\eqref{every tangent function lies above all tangent planes corollary} and \eqref{the inf restricts to balls condition}.

\begin{lem}\label{g is locally Lip and satisfies locally uniform supdiff estimates}
The function $g$ is locally Lipschitz, and for every $R>0$ there exists $C_R>0$ such that for every $x, h\in B(0,R)$ we have
$$
g(x+h)+g(x-h)-2g(x)\leq C_R|h|^2.
$$
\end{lem}
\begin{proof}
Given $R>0$, by \eqref{the inf restricts to balls condition} there exists $\eta=\eta(R)>0$ such that
$$
g(x)=\inf_{y\in E\cap B(0, \eta)}\{f(y)+\langle G(y), x-y\rangle +\varphi_{y}(x)\} \textrm{ for all } x\in B(0, R).
$$
Then, if $x, h\in B(0,R)$, for any given $\varepsilon>0$ we may find $y\in B(0, \eta)$ such that
\begin{equation}\label{estimate for g in the proof of locally uniform supdiff inequality}
g(x)\geq f(y)+\langle G(y), x-y\rangle+\varphi_y(x)-\varepsilon,
\end{equation}
and therefore, using the definition of $g$ (for the first inequality) and Taylor's theorem together with condition \eqref{finite sup of Lip constants of phiy on balls corollary} (for the second inequality), we obtain
\begin{align*}
g(x+h)& +g(x-h)-2g(x)  \leq f(y)+\langle G(y), x+h-y \rangle + \varphi_y(x+h) \\
& \quad + f(y)+\langle G(y), x-h-y \rangle + \varphi_y(x-h) \\
& \quad -2 \left( f(y)+ \langle G(y), x-y \rangle +  \varphi_y(x) \right) + 2 \varepsilon \\
& = \varphi_y(x+h) + \varphi_y(x-h)- 2\varphi_y(x) + 2 \varepsilon \\
& \leq  C_R |h|^2 +2 \varepsilon,
\end{align*}
where 
$C_R$ is given by condition \eqref{finite sup of Lip constants of phiy on balls corollary} applied with $\max\{2R, \eta(R)\}$ in place of $R$.
Since $\varepsilon>0$ is arbitrary, by sending $\varepsilon$ to $0$ we get what we need.
On the other hand, using again \eqref{estimate for g in the proof of locally uniform supdiff inequality}, we also have
\begin{eqnarray*}
& &g(x+h)-g(x)  \leq \\
& & f(y)+\langle G(y), x+h-y \rangle + \varphi_y(x+h)- \left( f(y)+ \langle G(y), x-y \rangle +  \varphi_y(x) \right) + \varepsilon \\
& & = \langle G(y), h\rangle + \varphi_y(x+h) -\varphi_y(x) + \varepsilon \\
& &  \leq  \left(\sup_{w\in B(0, \eta(R))}|G(w)|\right) |h| +\frac{1}{2}C_R|h|^2 + \varepsilon,
\end{eqnarray*}
which by letting $\varepsilon$ go to $0$ implies that
$$
g(x+h)-g(x) \leq  \left(\sup_{w\in B(0, \eta(R))}|G(w)|\right) |h| +\frac{1}{2}C_R|h|^2
$$
for all $x, h\in B(0, R)$. If $x, z\in B(0, R/2)$ and we take $h=z-x$ in this inequality, we obtain that
$$
g(z)-g(x) \leq \left(\sup_{w\in B(0, \eta(R))}|G(w)|\right) |z-x| +\frac{1}{2}C_R|z-x|^2,
$$
for all $x, z\in B(0, R/2)$. This implies that $g:\R^n\to\R$ is locally Lipschitz.
\end{proof}
Next we see that, under the standing assumptions, this kind of inequality is preserved (up to some constants) when we pass to the convex envelope.
\begin{lem}\label{the convex envelope is smooth}
Let $g:\R^n\to\R$ be a continuous function such that 
$
\lim_{|x|\to\infty}g(x)=\infty
$
and such that for every $R>0$ there exists $C_R>0$ so that for every $x, h\in B(0,R)$ we have
$$
g(x+h)+g(x-h)-2g(x)\leq C_R|h|^2.
$$
Then the function $F=\textrm{conv}(g)$ has a similar property: for every $R>0$ there exists $C'_R>0$ such that for every $x, h\in B(0,R)$ we have
$$
F(x+h)+F(x-h)-2F(x)\leq C'_R |h|^2.
$$
Therefore $F\in C^{1,1}_{\textrm{loc}}(\R^n)$.
\end{lem}
\begin{proof}
We will follow the proof of \cite{KirchheimKristensen} and make some appropriate changes. 
We may assume that 
\begin{equation}\label{g is positive}
g(z)\geq 0\textrm{ for all } z\in\R^n.
\end{equation}
Recall that an alternate expression for the convex envelope $F$ of a function $g:\R^n\to\R$ defined in \eqref{defn of convex envelope} is given by
\begin{equation}\label{alternate expression for the convex envelope}
F(x)=\inf\left\{\sum_{i=1}^{n+1}\lambda_i g(x_i) \, : \,  \lambda_i\geq 0, \, \sum_{i=1}^{n+1}\lambda_i=1, \, x=\sum_{i=1}^{n+1}\lambda_i x_i \right\}.
\end{equation}
Since $F\leq g$ by definition, and $g$ is bounded on bounded sets, so is $F$ (and in particular $F$ is well defined on all of $\R^n$). Then, since $\lim_{|x|\to\infty}g(x)=\infty$, we can find some $R'>R$ such that
\begin{equation}\label{first estimate for g}
g(z)\geq (n+1)\left( \sup_{y\in B(0, R)}F(y)+2\right) \textrm{ for all } z\in \R^n\setminus B(0, R').
\end{equation}
By applying the previous lemma with $(n+1)R'$ in place of $R$, we next find $C=C_{(n+1)R'}>0$ such that
\begin{equation}\label{second estimate for g}
g(z+v)+g(z-v)-2g(z)\leq C|v|^2 \textrm{ for all } z, v\in B(0, (n+1)R').
\end{equation}
Now, given $x, h\in B(0, R)$, we use \eqref{alternate expression for the convex envelope} to take a sequence $\{(\lambda_{i}^{(k)}, x_{i}^{(k)})_{1\leq i\leq n+1}\}_{k=1}^{\infty}$ such that
$$
\lambda_{1}^{(k)}\geq \lambda_{2}^{(k)} \geq ... \geq \lambda_{n+1}^{(k)}\geq 0, \,\,\, \sum_{i=1}^{n+1}\lambda_{i}^{(k)}=1, \,\,\, x=\sum_{i=1}^{n+1}\lambda_{i}^{(k)} x_{i}^{(k)}.
$$
and
\begin{equation}\label{approximation of F(x)}
F(x)=\lim_{k\to\infty}\sum_{i=1}^{n+1}\lambda_{i}^{(k)}g(x_{i}^{(k)}).
\end{equation}
Note that $$\lambda_{1}^{(k)}\geq\frac{1}{n+1}$$ for every $k$, and recall \eqref{g is positive}.
According to \eqref{approximation of F(x)}, there exists some $k_0\in\N$ such that if $k\geq k_0$ then
$$
\sum_{i=1}^{n+1}\lambda_{i}^{(k)} g(x_{i}^{(k)})<F(x)+1\leq\sup_{y\in B(0, R)}F(y)+1,
$$
which thanks to \eqref{g is positive} implies
$$
\frac{1}{n+1}g(x_{1}^{(k)})\leq\lambda_1  g(x_{1}^{(k)})<\sup_{y\in B(0, R)}F(y)+1.
$$
This inequality, together with \eqref{first estimate for g}, shows that
\begin{equation}
x_{1}^{(k)}\in B(0, R') \textrm{ for all } k\geq k_0.
\end{equation}
Therefore, up to extracting a subsequence, we may assume that these limits exist:
\begin{equation}
\lim_{k\to\infty}x_{1}^{(k)}:=x_1\in B(0, R'), \,\,\, \lim_{k\to\infty}\lambda_{1}^{(k)}:=\lambda_1\in [\tfrac{1}{n+1}, 1].
\end{equation}
Now we may write
$$
x+h=\lambda_{1}^{(k)}\left(x_{1}^{(k)}+\frac{h}{\lambda_{1}^{(k)}}\right) +\sum_{i=2}^{n+1}\lambda_{i}^{(k)} x_{i}^{(k)},
$$
and, because $F$ is convex and $F\leq g$, we have
\begin{align*}
& & F(x+h)=F\left(\lambda_{1}^{(k)}\left(x_{1}^{(k)}+\frac{h}{\lambda_{1}^{(k)}}\right) +\sum_{i=2}^{n+1}\lambda_{i}^{(k)} x_{i}^{(k)}\right) \\
& & \leq \lambda_{1}^{(k)}F\left(x_{1}^{(k)}+\frac{h}{\lambda_{1}^{(k)}}\right) +\sum_{i=2}^{n+1}\lambda_{i}^{(k)}F( x_{i}^{(k)} )\\
& & \leq \lambda_{1}^{(k)}g\left(x_{1}^{(k)}+\frac{h}{\lambda_{1}^{(k)}}\right) +\sum_{i=2}^{n+1}\lambda_{i}^{(k)}g( x_{i}^{(k)} ),
\end{align*}
which implies
\begin{equation}
F(x+h)-F(x)\leq \lambda_{1}^{(k)}\left(g\left(x_{1}^{(k)}+\frac{h}{\lambda_{1}^{(k)}}\right)-g(x_{1}^{(k)})\right) +\left(\sum_{i=1}^{n+1}\lambda_{i}^{(k)}g(x_{i}^{(k)}) -F(x)\right),
\end{equation}
and passing to the limit as $k\to\infty$ we get
\begin{equation}
F(x+h)-F(x)\leq \lambda_1\left( g(\left(x_1+\frac{h}{\lambda_1}\right)-g(x_1)\right).
\end{equation}
Similarly we obtain
\begin{equation}
F(x-h)-F(x)\leq \lambda_1\left( g(\left(x_1-\frac{h}{\lambda_1}\right)-g(x_1)\right).
\end{equation}
Thus we conclude, bearing in mind \eqref{second estimate for g} and the facts that $|h/\lambda_1|\leq (n+1)|h|\leq (n+1)R<(n+1)R'$ and $|x_1|\leq R'<(n+1)R'$, that
\begin{multline*}
F(x+h)+F(x-h)-2F(x) \\ \leq \lambda_1\left( g\left(x_1 +\frac{h}{\lambda_1}\right) + g\left(x_1-\frac{h}{\lambda_1}\right)-2g(x_1)\right)\lambda_1 C\left| \frac{h}{\lambda_1}\right|^2 \\
=\frac{1}{\lambda_1}C |h|^2\leq (n+1) C |h|^2.
\end{multline*}
We have shown that for every $R>0$ there exists $C'_R>0$ such that for every $x, h\in B(0, R)$, we have
$$
F(x+h)+F(x-h)-2F(x)\leq C'_R |h|^2.
$$
Since $F$ is convex, this is equivalent to saying that $F\in C^{1,1}_{\textrm{loc}}(\R^n)$, and in fact 
$$
\sup\left\{\frac{|\nabla F(x)-\nabla F(y)|}{|x-y|} : x, y\in B(0,R), x\neq y\right\}\leq C'_R
$$
(see, for instance, the proof of \cite[Proposition 2.2]{AzagraLeGruyerMudarra} restricted to a ball, and combine it with \cite[Corollary 3.3.8]{CannarsaSinestrari} or \cite[Theorem 1.5]{AzagraFerrera2015}).
\end{proof}

Let us now finish the proof of Theorem  \ref{Corollary to C11 Whitney thm for coercive convex functions}. Since $m$ is convex, by definition of convex envelope we have $$m \leq F \leq g \textrm{ on } \R^n,$$
which together with \eqref{g=f on E} allows us to conclude that $F=f$ on $E$.

Finally, we have $m \leq F$ on $\R^n$ and $F=m$ on $E$, where $m$ is convex and $F$ is differentiable on $\R^n$. This implies that $m$ is differentiable on $E$, with $\nabla m (x)= \nabla F (x) $ for all $x\in E$. Since we obviously have $G(x) \in \partial m (x)$ for all $x\in E$, we also obtain that $\nabla F(x) = G(x)$ for all $x\in E$. 
\qed

\medskip

\subsection{Proof of Theorem \ref{Corollary to C11 Whitney thm for coercive convex functions}, necessity.}
Let us assume that there exists a convex function $F\in C^{1,1}_{\textrm{loc}}(\R^n)$ such that $F(y)=f(y)$ and $\nabla F(y)=G(y)$ for all $y\in E$, and let us see that the functions $\varphi_y$, $y\in E$, defined by
\begin{equation}\label{definition of varphiy in necessity proof}
\varphi_y(x)=F(x)-F(y)-\langle \nabla F(y), x-y\rangle
\end{equation}
satisfy the conditions of Theorem \ref{Corollary to C11 Whitney thm for coercive convex functions}. Note that  $\nabla\varphi_y(x)=\nabla F(x)-\nabla F(y)$, so it is clear that \eqref{phiy equals 0 at y} holds true.
We also have, for every $x, y, z\in B(0, R)$, that
\begin{multline*}
\frac{|\nabla \varphi_{y}(x)-\nabla\varphi_y(z)|}{|x-z|}=
\frac{|\nabla F(x)-\nabla F(y)-(\nabla F(z)-\nabla F(y))|}{|x-z|} \\ =
\frac{|\nabla F(x)-\nabla F(z)|}{|x-z|}\leq \textrm{Lip}(\nabla F_{|_{B(0,R)}}),
\end{multline*}
so \eqref{finite sup of Lip constants of phiy on balls corollary} is also  satisfied. Besides, since $F$ is convex we have
\begin{equation}\label{F satisfies tangent functions condition}
F(z)+\langle \nabla G(z), x-z\rangle \leq F(x)=
F(y)+\langle \nabla F(y), x-y\rangle +\varphi_y(x) \textrm{ for all } x, y, z\in\R^n,
\end{equation} 
which implies \eqref{every tangent function lies above all tangent planes corollary}.  \qed

\medskip

\subsection{Proof of Theorem \ref{Corollary 2 to the main coercive result}.}
Although one can use condition \eqref{every tangent function lies above all tangent planes corollary 2} and standard techniques (smooth approximation and partitions of unity) to construct a family of functions $\{\varphi_y\}_{y\in E}$ as required to apply Theorem \ref{Corollary to C11 Whitney thm for coercive convex functions}, we prefer to use some tools of \cite{Azagra2013} so as to get a family of {\em convex} functions $\varphi_y$. Convexity of these functions is not needed in Theorem \ref{Corollary to C11 Whitney thm for coercive convex functions}, but we think that it may be useful in some other problems, and does not add any important complication in the proof of Theorem \ref{Corollary 2 to the main coercive result}.
\begin{lem}[Smooth maxima, see Lemma 1 of \cite{Azagra2013}]\label{smooth maxima}
For every $\delta>0$ there exists a $C^\infty$ function $\mathcal{M}_{\delta}:\R^{2}\to\R$ with the following properties:
\begin{enumerate}
\item $\mathcal{M}_{\delta}$ is convex;
\item $\max\{x, y\}\leq \mathcal{M}_{\delta}(x, y)\leq \max\{x,y\}+\frac{\delta}{2}$ for all $(x,y)\in\R^2$.
\item $\mathcal{M}_{\delta}(x,y)=\max\{x,y\}$ whenever $|x-y|\geq\delta$.
\item $\mathcal{M}_{\delta}(x,y)=\mathcal{M}_{\delta}(y,x)$.
\end{enumerate}
\end{lem}
\begin{proof}
It is easy to construct a $C^{\infty}$ function $\theta=\theta_{\delta}:\R\to (0,
\infty)$ such that:
\begin{enumerate}
\item $\theta(t)=|t|$ if and only if $|t|\geq\delta$;
\item $\theta$ is convex and symmetric;
\item $\textrm{Lip}(\theta)=1$.
\end{enumerate}
Then the function $\mathcal{M}_{\delta}$ defined by
$
\mathcal{M}_{\delta}(x,y)=\frac{1}{2}\left(x+y+\theta_{\delta}(x-y)\right)
$
has the required properties.
\end{proof}

These {\em smooth maxima} $\mathcal{M}_{\delta}$ are useful to approximate the
maximum of two functions without losing convexity or other key
properties of the functions, as in the following proposition.

\begin{prop}[See Proposition 2 of \cite{Azagra2013}]\label{properties of M(f,g)}
Let $\mathcal{M}_{\delta}$ be as in the preceding Lemma, and let $f, g: \R^n\to\R$
be convex functions. For every $\delta>0$, the function
$\mathcal{M}_{\delta}(f,g):\R^n\to\R$ has the following properties:
\begin{enumerate}
\item $\mathcal{M}_{\delta}(f,g)$ is convex.
\item If $f, g$ are of class $C^k$, then so is $\mathcal{M}_{\delta}(f,g)$.
\item $\mathcal{M}_{\delta}(f,g)=f$ if $f\geq g+\delta$.
\item $\mathcal{M}_{\delta}(f,g)=g$ if $g\geq f+\delta$.
\item $\max\{f,g\}\leq \mathcal{M}_{\delta}(f,g)\leq \max\{f,g\} + \delta/2$.
\item $\mathcal{M}_{\delta}(f,g)=\mathcal{M}_{\delta}(g, f)$.
\item $\textrm{Lip}(\mathcal{M}_{\delta}(f,g)_{|_B})\leq \max\{ \textrm{Lip}(f_{|_B}), \textrm{Lip}(g_{|_B}) \}$ for every ball $B\subset \R^n$.
\item If $f_1\leq f_2$ and $g_1\leq g_2$ then $\mathcal{M}_{\delta}(f_1, g_1)\leq \mathcal{M}_{\delta}(f_2, g_2)$.
\item If $f, g\in C^{2}(\R^n)$ then, for each ball $B\subset\R^n$,
\begin{multline*}
\sup_{x\in B}\|D^2 \mathcal{M}_{\delta}(f,g)(x)\| \\ \leq C_{\delta} \left( \sup_{x\in B}\|D^2 f(x)\|+ \sup_{x\in B}\|D^2 g(x)\| +\left(\textrm{Lip}(f_{|_B})+\textrm{Lip}(g_{|_B})\right)^2\right),
\end{multline*}
where $C_{\delta}>0$ is a constant depending only on $\delta$.
\end{enumerate}
\end{prop}
\begin{proof}
See \cite{Azagra2013} for properties $(1)-(8)$.
To check $(9)$, it is sufficient to see that the function
$t\mapsto \mathcal{M}_{\delta}(f,g)(\gamma(t))$ has a suitably bounded
second derivative, where $\gamma(t)=x+tv$ with $\|v\|=1$. So, by replacing $f, g$
with $f(\gamma(t))$ and $g(\gamma(t))$ we can assume that $f$ and
$g$ are defined on an interval $I\subseteq\R$. In this case we easily compute
\begin{multline*}
\frac{d^2}{dt^2}\mathcal{M}_{\delta}(f(t), g(t))=
\frac{ \left(1+\theta'(f(t)-g(t))\right) f''(t) +
\left(1-\theta'(f(t)-g(t))\right) g''(t)}{2} \\
 + \frac{\theta''(f(t)-g(t)) \left( f'(t)-g'(t)\right)^{2}}{2},
\end{multline*}
and the estimate of $(9)$ follows immediately.
\end{proof}

Now we can prove Theorem \ref{Corollary 2 to the main coercive result}. For each $k\in\N$, we denote
$
B_k:=B(0, k).
$
By the main result of \cite{Azagra2013}, we may find a $C^{\infty}$ convex function $\psi:\R^n\to\R$ such that
\begin{equation}\label{psi approximates m}
m(x)\leq \psi(x)\leq m(x)+\frac{1}{2} \textrm{ for all } x\in\R^n
\end{equation}
where $m(x):=\sup_{z\in E}\{f(z)+\langle G(z), x-z\rangle\}$.
In particular $\psi\in C^{1,1}_{\textrm{loc}}(\R^n)$. Next, for each $y\in E$, we define $k(y)$ as the first positive integer $k$ such that $y\in B_{k}$, and the function $\varphi_{y}:\R^n\to\R$ by
$$
\varphi_{y}(x)=
\begin{cases}
 A_{k(y)}|x-y|^2 & \textrm{ if } x\in B_{3k(y)}, \\
 \mathcal{M}_{1/8}\left( A_{k(y)}|x-y|^2 , \, \psi(x)-f(y)-\langle G(y), x-y\rangle\right) & \textrm{ if } x\notin B_{3k(y)},
\end{cases}
$$
where $\mathcal{M}_{1/8}$ is the smooth maximum $\mathcal{M}_{\delta}$ of Lemma \ref{smooth maxima} with $\delta=1/8$, and the numbers $A_k\geq 2$ are given by condition \eqref{every tangent function lies above all tangent planes corollary 2}.
By replacing $A_k$ with $\max_{1\leq j\leq k}A_j$ if necessary, we may assume that
\begin{equation}
A_{k}\leq A_{k+1} \textrm{ for all } k\in\N.
\end{equation}
Note that, if $x\in B_{4k(y)}\setminus B_{2k(y)}$ then  $|x-y|\geq 1$, so we have, using \eqref{every tangent function lies above all tangent planes corollary 2}, that
\begin{multline*}
A_{k(y)} |x-y|^2\geq \frac{1}{2}A_{k(y)} |x-y|^2 +\frac{1}{2}A_{k(y)} 
\geq m(x)-f(y)-\langle G(y), x-y\rangle + \frac{1}{2}A_{k(y)}\\
\geq \psi(x)-\frac{1}{2} -f(y)-\langle G(y), x-y\rangle + \frac{1}{2}A_{k(y)}
\geq \psi(x)-f(y)-\langle G(y), x-y\rangle +\frac{1}{8},
\end{multline*}
which implies, by Proposition \ref{properties of M(f,g)}(3), that 
$$
\mathcal{M}_{1/8}\left( A_{k(y)}|x-y|^2 , \, \psi(x)-f(y)-\langle G(y), x-y\rangle\right)=A_{k(y)} |x-y|^2
$$
for all $x\in B_{4k(y)}\setminus B_{2k(y)}$.
We then easily deduce (bearing in mind the definition of $\varphi$ and Proposition \ref{properties of M(f,g)}) that $\varphi_{y}\in C^{\infty}(\R^n)$ and $\varphi_y$ is convex.

Let us see that the $1$-jet $(f(y), G(y))_{y\in E}$, together with the family $\{\varphi_{y}\}_{y\in E}$, satisfy the properties of Theorem \ref{Corollary to C11 Whitney thm for coercive convex functions}.
Property \eqref{phiy equals 0 at y} is obvious. Let us check property \eqref{every tangent function lies above all tangent planes corollary}. Given a point $y\in E$, recall that $k(y)$ is the first $k\in\N$ such that $y\in B_k$. If $x\in B_{3k(y)}$ then by the definitions of $A_{k(y)}$ and $\varphi_y$ we have
$$
m(x)-f(y)-\langle G(y), x-y\rangle\leq \frac{1}{2}A_{k(y)} |x-y|^2 \leq 
A_{k(y)}|x-y|^2=\varphi_y(x).
$$
On the other hand, if $x\notin B_{3k(y)}$ then
\begin{eqnarray*}
& & \varphi_y (x)=\mathcal{M}_{1/8}\left( A_{k(y)}|x-y|^2 , \, \psi(x)-f(y)-\langle G(y), x-y\rangle\right) \\
& & \geq \psi(x)-f(y)-\langle G(y), x-y\rangle\geq 
m(x)-f(y)-\langle G(y), x-y\rangle,
\end{eqnarray*}
where we have used Proposition \ref{properties of M(f,g)}(5) for the first inequality and \eqref{psi approximates m} for the second one. In either case we see that
\begin{equation}\label{tangent functions are above f}
m(x) \leq f(y)+\langle G(y), x-y\rangle + \varphi_y(x) \textrm{ for all } y\in E, x\in\R^n,
\end{equation}
which is equivalent to condition \eqref{every tangent function lies above all tangent planes corollary}.

Let us now verify \eqref{finite sup of Lip constants of phiy on balls corollary}. Since $\varphi_{y}\in C^{2}(\R^n)$ for every $y\in E$, this amounts to showing that 
$$
\sup\{ \|D^{2}\varphi_y (x)\| \, : \, x\in B(0,R), y\in E\cap B(0,R)\}<\infty \textrm{ for every } R>0,
$$
or equivalently for every $R\in\N$.
Given $R\in\N$ and $y\in E\cap B(0, R)$, note that $k(y)\leq R$. If $x\in B(0, R)\cap B_{3k(y)}$ then
\begin{equation}\label{estimate for the second derivative of phiy on balls 3k}
\|D\varphi_{y}(x)\|=2A_{k(y)}.
\end{equation}
On the other hand, if $x\in B(0, R)$ then 
\begin{equation}
\| D^2_{x}\left( \psi(x)-f(y)-G(y), x -y\rangle\right)\|=\|D^{2}\psi(x)\|\leq
\sup_{z\in B(0, R)}\|D^2\psi(z)\|.
\end{equation}
Using these estimates with Proposition \ref{properties of M(f,g)}(9) we obtain that
\begin{eqnarray*}
& & \sup_{x\in B(0, R)}\|D^2_{x}\mathcal{M}_{1/8}\left( A_{k(y)}|x-y|^2 , \, \psi(x)-f(y)-\langle\nabla f(y), x-y\rangle\right)\| \\
& & \leq
C_{1/8}\left( 2A_{k(y)} + \sup_{z\in B(0, R)}\|D^2\psi(z)\|+\left(\textrm{Lip}\left(\psi_{|_{B(0, R)}}\right) +2A_{k(y)}\right)^2\right),
\end{eqnarray*}
and by combining this inequality with \eqref{estimate for the second derivative of phiy on balls 3k}, and bearing in mind the definition of $\varphi_y$ and the facts that $k(y)\leq R$ and the sequence $\{A_k\}$ is increasing, we obtain
\begin{equation}
\sup_{x\in B(0, R)}\|D^2\varphi_y(x)\| \leq C \left( A_R + \sup_{z\in B(0, R)}\|D^2\psi(z)\|+\left(\textrm{Lip}\left(\psi_{|_{B(0, R)}}\right) +A_R\right)^2\right),
\end{equation}
where $C$ is an absolute constant. This shows \eqref{finite sup of Lip constants of phiy on balls corollary}.
Thus we have proved the sufficiency part of Theorem \ref{Corollary 2 to the main coercive result}.

The necessity part is obvious: just take $A_k=\textrm{Lip}\left( (\nabla F)_{|_{B(0, 4k)}}\right)$.
\qed

\medskip

\subsection{Proof of the extension properties of \eqref{extension formula for G bounded 1}.} We keep denoting $B_k=B(0, k)$. Given $y\in E$, if $x\in B_{2k(y)}$ then, by condition \eqref{every tangent function lies above all tangent planes corollary 2} we have
$$
m(x)\leq f(y)+\langle G(y), x-y\rangle +\frac{A_{k(y)}}{2}|x-y|^2.
$$
On the other hand, if $x\notin B_{2k(y)}$, then $|x-y|\geq 1$, and, observing that $m(y)=f(y)$ and $\textrm{Lip}(m)=\|G\|_{\infty}$, we have
\begin{eqnarray*}
& & m(x)-f(y)-\langle G(y), x-y\rangle = m(x)-m(y) + \langle G(y), x-y\rangle\\
& &\leq 2\|G\|_{\infty} |x-y|\leq 2\|G\|_{\infty} |x-y|^2.
\end{eqnarray*}
In either case we have
$$
m(x) \leq f(y)+\langle G(y), x-y\rangle + \left( \frac{A_{k(y)}}{2} + 2\|G\|_{\infty}\right)|x-y|^2 \textrm{ for all } x\in\R^n, y\in E,
$$
that is condition \eqref{every tangent function lies above all tangent planes corollary} is satisfied with $\varphi_{y}(x)= \left( \frac{A_{k(y)}}{2} + 2\|G\|_{\infty}\right)|x-y|^2$. It is clear that the rest of the conditions of Theorem \ref{Corollary to C11 Whitney thm for coercive convex functions} are met as well, and by taking $a=1/2$ it follows that \eqref{extension formula for G bounded 1} is an extension of the jet $(f, G)$.   \qed

\medskip

\subsection{Proof of Theorem \ref{MainTheorem with P}: necessity.} 
Assume that there exists a convex function $F\in C^{1,1}_{\textrm{loc}}(\R^n)$ such that $F=f$, $\nabla F=G$ on $E$, and $X=X_{F}=\textrm{span}\{\nabla F(y)-\nabla F(z) : y, z\in \R^n\}$. By Theorem \ref{rigid global behavior of convex functions}, there is a unique $C^{1,1}_{\textrm{loc}}$ convex function $c:X\to\R$ and a unique vector $v\in X^{\perp}$ such that we have the decomposition 
\begin{equation}\label{decomposition of F in necessity main thm}
F=c\circ P+\langle v, \cdot\rangle,
\end{equation}
which implies 
\begin{equation}\label{decomposition of the gradient of F in necessity main thm}
\nabla F=\nabla c\circ P +v.
\end{equation}
Let us check that properties $(i)-(iii)$ of Theorem \ref{MainTheorem with P} are satisfied for $f=F_{|_E}$ and $G=(\nabla F)_{|_E}$.

$(i)$: This is obvious.

$(ii)$: Assume that $Y:=\textrm{span}\{\nabla F(x)-\nabla F(y) : x, y\in E\}$ is strictly contained in $X$. With $k$ and $d$ denoting the dimensions of $Y$ and $X$ respectively, we can find points $x_0,x_1, \ldots, x_k \in E$ such that $Y =\textrm{span} \lbrace \nabla F(x_j)-\nabla F(x_0) \: : \: j=1, \ldots, k \rbrace.$ Then there must exist $p_1 \in \R^n$ such that $\nabla F(p_1)-\nabla F(x_0) \notin Y$ (otherwise we would have that $\nabla F(p)-\nabla F(x_0) \in Y$ for all $p\in \R^n,$ which implies that
$$
\nabla F(p)-\nabla F(q)=(\nabla F(p)-\nabla F(x_0)) - (\nabla F(q)-\nabla F(x_0)) \in Y, \quad \text{for all} \quad p,q \in \R^n,
$$
contradicting that $X \neq Y$). Then the subspace $Y_1$ spanned by $Y$ and the vector $\nabla F(p_1)-\nabla F(x_0) $ has dimension $k+1.$ If $d=k+1,$ we are done. Otherwise we repeat this argument and by induction we obtain points $p_1, \ldots ,p_{d-k} \in \R^n$ such that the set $ \lbrace \nabla F(p_j)- \nabla F(x_0) \rbrace_{j=1}^{d-k}$ is linearly independent and $X = Y \oplus \textrm{span}\lbrace \nabla F(p_j)- \nabla F(x_0) \: : \: j=1, \ldots ,d-k \rbrace$, concluding that
$
X=\textrm{span}\left\{u-w : u, w\in \nabla F(E^{*}) \right\},
$
where $E^{*}=E\cup\{p_1, ..., p_{d-k}\}$.

Now let us define, for each $y\in E^{*}$, the function $\varphi_y :X\to\R$ by
\begin{equation}\label{defn of varphiy in necessity main thm}
\varphi_y(x)=c(x)-c(P(y))-\langle \nabla c(P(y)), x-P(y)\rangle,
\end{equation}
where $c$ is as in \eqref{decomposition of F in necessity main thm}. It is clear that $\varphi_y$ satisfies \eqref{phiy equals 0 at y corollary}, which because $\varphi_y$ is convex implies that $\varphi_y(x)\geq 0$ for all $x\in X$. For each $y\in E^{*}$, let us denote $t_y=F(y)$ and $\xi_y=\nabla F(y)$. Note that, as $\nabla c(P(y))\in X$, we have, for every $x\in\R^n$, $y\in E^{*}$,
\begin{multline*}
F(x)-F(y)-\langle \nabla F(y), x-y\rangle \\ = c(P(x))+\langle v, x\rangle-c(P(y))-\langle v, y\rangle -\langle \nabla c(P(y)) +v, x-y\rangle \\
=c(P(x))-c(P(y))-\langle \nabla c(P(y)), x-y\rangle
 \\ = c(P(x))-c(P(y))-\langle \nabla c(P(y)), P(x-y)\rangle
 =\varphi_y(P(x)).
\end{multline*}
Therefore, since $F$ is convex and $(F, \nabla F)=(f, G)$ on $E$, and by the definition of $t_y, \xi_y$, we have, for every $x\in \R^n$, $y\in E^{*}$,
\begin{eqnarray*}
& & t_{y}+\langle\xi_y, x-y\rangle +\varphi_{y}(P(x))=F(x)=
\sup_{z\in\R^n}\{F(z)+\langle\nabla F(z), x-z\rangle\}\\
& &\geq \sup_{z\in E^{*}}\{F(z)+\langle\nabla F(z), x-z\rangle\} =: m^{*}(x),
\end{eqnarray*}
so \eqref{every tangent function lies above all tangent planes MainthmwithP} holds true. Finally, since $\nabla\varphi_y(u)=\nabla c(u)-\nabla c(P(y))$ and $c\in C^{1,1}_{\textrm{loc}}(X)$, we have that
\begin{eqnarray*}
& &\sup\left\{ \frac{\nabla \varphi_y(u)-\nabla\varphi_y(w)}{|u-w|} : y\in E\cap P^{-1}(B_X(0,R)), u, w\in B_{X}(0,R), u\neq w \right\} \\
=
& & \sup\left\{ \frac{\nabla c(u)-\nabla c(w)}{|u-w|} : u, w\in B_{X}(0,R), u\neq w \right\}=\textrm{Lip}\left( (\nabla c)_{|_{B_X(0,R)}}\right)<\infty
\end{eqnarray*}
for each $R>0$, and \eqref{finite sup of Lip constants of phiy on balls main thm} is satisfied as well. 

$(iii)$ In this case there is no need to add new data, and the same proof works with $E^{*}=E$.
\qed

\medskip

\subsection{Proof of Theorem \ref{MainTheorem with P}: sufficiency.}
Consider the function
$$
m^{*}(x) :=\sup_{y\in E^{*}}\{t_y +\langle\xi_y, x-y\rangle\}.
$$
\begin{lem}
The function $m^{*}:\R^n\to\R$ is  well defined, convex, and satisfies
$$
m^{*}(y)=t_y \,\,\, \textrm{ and } \xi_y\in\partial m^{*}(y) \, \textrm{ for every } y\in E^{*}.
$$
and, with the notation of Theorem \ref{rigid global behavior of convex functions},
$X=X_{m^{*}}$.
\end{lem}
\begin{proof}
By \eqref{every tangent function lies above all tangent planes MainthmwithP} we have, for any $y_0\in E$,
$$
m^{*}(x)\leq t_{y_0}+\langle\xi_{y_0}, x-y_0\rangle +\varphi_{y_0}(P(x)),
$$
so it is clear that $m^{*}(x)<\infty$ for every $x\in\R^n$. Obviously $m^{*}$ is convex, and using that $\varphi_y(P(y))=0$ it is easily checked that $m^{*}(x)=t_y$, which immediately implies that $m^{*}(x)\geq t_{y}+\langle\xi_y, x-y\rangle$ for all $x\in\R^n$, that is, $\xi_y\in\partial m^{*}(x)$. 

Let us check that $X=X_{m^{*}}$. By assumption $X= \textrm{span} \left( \lbrace \xi_y-\xi_z \: : \: y,z \in E^* \rbrace \right).$ On the one hand, we have that $m^{*}$ is essentially coercive in the direction of $X$. Indeed, if $X=\{0\}$ then $m$ is affine and the result is obvious; therefore we can assume $\textrm{dim}(X)\geq 1$ and find points $y_0, y_1, \ldots, y_k\in E$ such that $\{v_1, \ldots, v_k\}$ is a basis of $X$, where
$$
v_j=G(y_j)-G(y_0), \,\,\, j= 1, \ldots, k.
$$
Then, with the terminology of \cite[Section 4]{Azagra2013},  we have that 
$$C(x)=\max\{ t_{y_0}+\langle\xi_{y_0}, x-y_0\rangle, \, t_{y_1}+\langle \xi_{y_1}, x-y_1\rangle, \ldots, \, t_{y_k}+\langle \xi_{y_k}, x-y_k\rangle\}
$$
is a {\em $k$-dimensional corner function} such that $C(x)\leq m^{*}(x)$ for all $x\in\R^n$. This implies that $C$ is essentially coercive in the direction of $X$, hence so is $m^{*}$, and by Theorem \ref{rigid global behavior of convex functions} we infer that $X\subseteq X_{m^{*}}$.
 
On the other hand, if $X_{m^{*}}\neq X$, we can take a vector $w\in X_{m^{*}}\setminus\{0\}$ with  $w\perp X$, and then we obtain, for all $t\in\R$,
\begin{eqnarray*}
& & m^{*}(y_0+tw)-t_{y_0}-\langle \xi_{y_0}, tw\rangle=	
\sup_{z\in E^{*}}\{ t_z- t_{y_0}+\langle \xi_{z}-\xi_{y_0}, tw\rangle +\langle \xi_z, y_0-z\rangle \} \\
& & =\sup_{z\in E^{*}}\{ t_z- t_{y_0} +\langle \xi_z, y_0-z\rangle\}=m^{*}(y_0)-t_{y_0}=0,
\end{eqnarray*}
hence the function $\R\ni t\mapsto m^{*}(x_0+tw)$ cannot be essentially coercive, contradicting the assumption that $w\in X_{m^{*}}$.
\end{proof}
By applying Theorem \ref{rigid global behavior of convex functions} to the function
$m^{*}$, and using the preceding lemma,
we can write
\begin{equation}
m^{*}=c^{*}\circ P +\langle v, \cdot\rangle.
\end{equation}
Now let us define $$E^{\flat}=P(E^{*})\subset X,$$ and $f^{\flat}:E^{\flat}\to\R$, $G^{\flat}:E^{\flat}\to\ X$ by
\begin{equation}\label{defn of fflat and Gflat}
f^{\flat}(z)=c^{*}(z), \quad \quad G^{\flat}(z)=\xi_{y}-v, \textrm{ where } y\in P^{-1}(z).
\end{equation}
Note that if $y, y' \in P^{-1}(z)$ then $\xi_{y}=\xi_{y'}$, as otherwise, according to Theorem \ref{rigid global behavior of convex functions} and the facts that $\xi_y\in\partial m^{*}(y)$ and $\xi_{y'}\in\partial m^{*}(y')$, the function $m^{*}$ would be essentially coercive in the direction $\textrm{span}\{y-y'\}$, which is perpendicular to  $X$, contradicting that $X_{m^{*}}=X$. Therefore the function $G^{\flat}$ is well defined.

Let us also define the functions $m^{\flat}, g^{\flat}:X\to\R$ by
$$
m^{\flat}=c^{*},
$$
and
$$
g^{\flat}(x)=\inf\left\{ f^{\flat}(z)+\langle G^{\flat}(z), x-z\rangle +\varphi_{y}(x)+a|x-z|^2 \, : \, z\in E^{\flat}, y\in P^{-1}(z)\right\}
$$
for all $x\in X$,
where $a>0$ is a given number.
\begin{lem}
We have that
$$
m^{\flat}(x)\leq g^{\flat}(x) \textrm{ for every } x\in X,
$$
and
$$
m^{\flat}(z)=f^{\flat}(z), \,\,\,  \{G^{\flat}(z)\}=\partial m^{\flat}(z)
 \textrm{ for every } z\in E^{\flat}.
$$
\end{lem}
\begin{proof}
Since $v\in X^{\perp}$, we have that $m^{*}(z)=c^{*}(z)$ for all $z\in X$, and this implies that $m^{\flat}(z)=f^{\flat}(z)$ whenever $z\in E^{\flat}$. On the other hand, using \eqref{every tangent function lies above all tangent planes MainthmwithP} and the facts that $G^{\flat}(z)\in X$ and $v\in X^{\perp}$, we have, for every $y\in P^{-1}(z)$, $z\in E^{\flat}$, and $x\in X$,
\begin{eqnarray*}
& & m^{\flat}(x)=c^{*}(x)=m^{*}(x)\leq t_y +\langle \xi_y, x-y\rangle +\varphi_y (Px)+ a|P(x-y)|^2 \\
& &= t_{y}+\langle v+G^{\flat}(z), x-y\rangle +\varphi_y(x)+ a|P(x-y)|^2 \\
& &= c^{*}(P(y))+\langle v, y\rangle +\langle v+G^{\flat}(z), x-y\rangle +\varphi_y(x) + a|P(x-y)|^2 \\
& &= c^{*}(z)+\langle G^{\flat}(z), x-y\rangle  + \varphi_y(x) + a|P(x-y)|^2 \\
& &= f^{\flat}(z)+\langle G^{\flat}(z), x-z\rangle + \varphi_y(x) + a|x-z|^2.
\end{eqnarray*}
By taking the infimum over such $z, y$, we obtain that $m^{\flat}(x)\leq g^{\flat}(x)$ for all $x\in X$. Since $m^{\flat}$ is convex, since the function $\varphi_y$ is differentiable, and the above inequality becomes an equality for $x=z\in E^{\flat}$, this inequality also shows that $m^{\flat}$ is differentiable at each $z\in E^{\flat}$, with $\nabla m^{\flat}(z)=G^{\flat}(z)$.
\end{proof}

Now we can repeat the steps of the proof of Theorem \ref{Corollary to C11 Whitney thm for coercive convex functions} with $m^{\flat}$ and $g^{\flat}$ in place of $m$ and $g$, respectively. As in that proof, \eqref{every tangent function lies above all tangent planes MainthmwithP} the preceding lemma implies that $f^{\flat}$ and $G^{\flat}$ are bounded on bounded sets.
By using Theorem \ref{rigid global behavior of convex functions} again, we also have that $m^{\flat}$ is essentially coercive on $X$ because
\begin{eqnarray*}
& & \textrm{span}\{\xi-\nu : \xi\in\partial m^{\flat}(z), \nu\in\partial m^{\flat}(z'), z, z'\in X\}\\
& & \supseteq
\textrm{span}\{G^{\flat}(z)-G^{\flat}(z') : z, z'\in E^{\flat}\}=
\textrm{span}\{\xi_y -\xi_{y'} : y, y'\in E\}=X.
\end{eqnarray*}
Thus, as in the proof of Theorem \ref{Corollary to C11 Whitney thm for coercive convex functions}, we may assume without loss of generality that 
\begin{equation}\label{mflat is assumed coercive}
\lim_{|x|\to\infty} m^{\flat}(x)=\infty, \,\,\, \textrm{ and } \,\,\, m^{\flat}(x)\geq 0 \textrm{ for all } x\in X,
\end{equation}
and we easily see that
\begin{equation}\label{gflat=fflat on E}
g^{\flat}(x) = m^{\flat}(x) = f^{\flat}(x) \textrm{ for all } x\in E^{\flat}.
\end{equation}
\begin{lem}\label{gflat restricts}
For every $R>0$ there exists $\eta=\eta(R)>0$ such that 
$$
g^{\flat}(x)=\inf\{ f^{\flat}(z)+\langle G^{\flat}(z), x-z\rangle +\varphi_{y}(x)+a|x-z|^{2} \, : \, z\in E^{\flat}\cap B_{X}(0, \eta), y\in P^{-1}(z)\}
$$
for all $x\in B_{X}(0, R)$.
\end{lem}
\begin{proof}
For each $y\in P^{-1}(z)$ with $z\in E^{\flat}$, we write
\begin{equation}\label{definition of varphiy tilde in main proof}
\widetilde{\varphi}_y(x)=\varphi_y(x)+a|x-z|^2.
\end{equation}
By using \eqref{finite sup of Lip constants of phiy on balls main thm} we see that
$$
\widetilde{M}_R:= 
\sup\left\{ \frac{\nabla {\widetilde{\varphi}}_y(u)-\nabla {\widetilde{\varphi}}(w)}{|u-w|} : y\in E\cap P^{-1}(B_X(0,R)), u, w\in B_{X}(0,R), u\neq w \right\}
$$
is finite for every $R>0$. Take $R_0>0$ so that $E^{\flat}\cap B_X(0, R_0)$ is nonempty, fix a point $z_0\in E^{\flat}\cap B_X(0, R_0)$, and for any given $R\geq R_0$ note that, for every $y_0\in P^{-1}(z_0)$,
$$
\widetilde{\varphi}_{y_0}(x) \leq \frac{(R+R_0)^2 \widetilde{M}_{R}}{2} \textrm{ for every } x\in B(0, R).
$$
Setting
$$
\eta=\eta(R):= R+ (R+R_0)\sqrt{\widetilde{M_R}/2a},
$$
we have that, for every $y\in P^{-1}(z)$ with $z\in E^{\flat}\setminus B_X(0, \eta)$ and every $x\in B_X(0,R)$,
\begin{eqnarray*}
& & f^{\flat}(z)+\langle G^{\flat}(z), x-z\rangle + \widetilde{\varphi}_y(x) \geq  f^{\flat}(z_0)+\langle G^{\flat}(z_0), x-z_0\rangle +a |x-z|^2 \\
& &\geq f^{\flat}(z_0)+\langle G^{\flat}(z_0), x-z_0\rangle + a\left(\eta-R\right)^2 \\
& &\geq f^{\flat}(z_0)+\langle G^{\flat}(z_0), x-z_0\rangle + \frac{(R+R_0)^2 \widetilde{M}_{R}}{2} \\
& &\geq f^{\flat}(z_0)+\langle G^{\flat}(z_0), x-z_0\rangle +\widetilde{\varphi}_{y_0}(x),
\end{eqnarray*}
which implies that the infimum in the definition of $g^{\flat}(x)$ can be restricted as stated.
\end{proof}

\begin{lem}\label{estimate for supdiff of gflat in main thm proof}
The function $g^{\flat}$ is locally Lipschitz, and for every $R>0$ there exists $C_R>0$ such that for every $x, h\in B(0,R)$ we have that
$$
g^{\flat}(x+h)+g^{\flat}(x-h)-2g^{\flat}(x)\leq C_R|h|^2.
$$
\end{lem}
\begin{proof}
Given $R>0$, we take $\eta(R)$ as in the preceding lemma, and we have that, if $x, h\in B_X(0,R)$, for any given $\varepsilon>0$ there exist $z\in B_X(0, \eta)$ and $y\in P^{-1}(z)$ such that
$$
g^{\flat}(x)\geq f^{\flat}(z)+\langle G^{\flat}(z), x-z\rangle+\varphi_y(x)-\varepsilon.
$$
Then, using the definition of $g$ and Taylor's theorem, we obtain
\begin{multline*}
g^{\flat}(x+h) +g^{\flat}(x-h)-2 g^{\flat}(x)
\leq f^{\flat}(z)+\langle G^{\flat}(z), x+h-y \rangle + \varphi_y(x+h)+a|x-z+h|^2 \\
\quad + f^{\flat}(z)+\langle G^{\flat}(z), x-h-y \rangle + \varphi_y(x-h)+ a|x-z-h|^2 \\
\quad -2 \left( f^{\flat}(z)+ \langle G^{\flat}(z), x-y \rangle +  \varphi_y(x) +a|x-z|^2\right) + 2 \varepsilon \\
 = \varphi_y(x+h) + \varphi_y(x-h)- 2\varphi_y(x) + 2a|h|^2+ 2 \varepsilon \\
 \leq  (K_R +2a) |h|^2 +2 \varepsilon,
\end{multline*}
where 
$K_R\in (0, \infty)$ is given by condition \eqref{finite sup of Lip constants of phiy on balls main thm} applied with $\max\{2R, \eta(R)\}$ in place of $R$.
Since $\varepsilon>0$ is arbitrary, by sending $\varepsilon$ to $0$ we obtain the desired estimate. One can also see that $g^{\flat}$ is locally Lipschitz as in the proof of Lemma \ref{g is locally Lip and satisfies locally uniform supdiff estimates}.
\end{proof}

Next let us define $F^{\flat}:X\to\R$ by
\begin{multline*}
F^{\flat}=\textrm{conv}_{X}(g^{\flat})= \\
\textrm{conv}_{X}\left( x\mapsto  
\inf_{z\in E^{\flat}, y\in P^{-1}(z)}\left\{ f^{\flat}(z)+\langle G^{\flat}(z), x-z\rangle +\varphi_y(x)+a|x-z|^2\right\}\right),
\end{multline*}
where $\textrm{conv}_{X}(\varphi)$ denotes the convex envelope of a function $\varphi:X\to\R$.
\begin{lem}
For every $R>0$ there exists $C'_R>0$ such that for every $x, h\in B(0,R)$ we have
$$
F^{\flat}(x+h)+F^{\flat}(x-h)-2F^{\flat}(x)\leq C'_R |h|^2.
$$
Therefore $F^{\flat}\in C^{1,1}_{\textrm{loc}}(X)$.
\end{lem}
\begin{proof}
Use Lemma \ref{the convex envelope is smooth} with $X$, $g^{\flat}$, and $F^{\flat}$ in place of $\R^n$, $g$, and $F$.
\end{proof}
Since $m^{\flat}$ is convex, we have $m^{\flat} \leq F^{\flat} \leq g^{\flat}$ on $X$,
which together with \eqref{gflat=fflat on E} yields $F^{\flat}=f^{\flat}$ on $E$. By the same argument as in the proof of Theorem \ref{Corollary to C11 Whitney thm for coercive convex functions}, we also have $\nabla F^{\flat}(z) = G^{\flat}(z)$ for all $z\in E^{\flat}$. 

Finally, let us define $F:\R^n\to\R$ by
\begin{equation}\label{defn of F in proof of main thm}
F(x)=F^{\flat}(P(x))+\langle v, x\rangle.
\end{equation}
Note that, if $y\in E\subset E^{*}$ then
$$
F(y)=F^{\flat}(P(y))+\langle v, y\rangle=f^{\flat}(P(y))+\langle v, y\rangle = c^{*}(P(y))+\langle v, y\rangle =m^{*}(y)=t_y=f(y),
$$
and also, according to \eqref{defn of fflat and Gflat},
$$
\nabla F(y)=\nabla F^{\flat}(P(y))+v=G^{\flat}(P(y))+v=\xi_y=G(y).
$$
Therefore $(F, \nabla F)$ extends $(f, G)$ from $E$ to $\R^n$. 

Let us also see that $F$ agrees with the expression given by \eqref{formula for F in main thm}. To do so, we use the following fact, whose proof is simple and can be omitted.
\begin{lem}
If $P:\R^n\to X$ is an orthogonal projection and $\psi:X\to\R$ then
$$
\textrm{conv}_{\,\,\R^n}(\psi\circ P)=(\textrm{conv}_{\,X}(\psi))\circ P.
$$
\end{lem}
Given $x\in \R^n$, $z\in E^{\flat}$, $y\in P^{-1}(z)$, we have
\begin{eqnarray*}
& & t_y +\langle \xi_y, x-y\rangle +\varphi_y (P(x))+ a|P(x-y)|^2 \\
& &= c^{*}(P(y))+\langle v, y\rangle +\langle v+G^{\flat}(z), x-y\rangle +\varphi_y(P(x)) + a|P(x)-z|^2 \\
& &=  f^{\flat}(z)+\langle G^{\flat}(z), P(x)-z\rangle +\langle v, x\rangle + \varphi_y(P(x)) + a|P(x)-z|^2.
\end{eqnarray*}
This implies that
\begin{eqnarray*}
\inf_{y\in E^{*}}\{t_y +\langle \xi_y, x-y\rangle +\varphi_y (P(x))+ a|P(x-y)|^2\}=
g^{\flat}(P(x))+\langle v, x\rangle,
\end{eqnarray*}
and by taking convex envelopes and using the preceding lemma we conclude that
\begin{multline*}
\textrm{conv}_{\R^n}\left( x\mapsto 
\inf_{y\in E^{*}}\{t_y +\langle \xi_y, x-y\rangle +\varphi_y (P(x))+ a|P(x-y)|^2\}\right) \\ =
F^{\flat}\circ P +\langle v, \cdot\rangle =F.
\end{multline*}
The proof of Theorem \ref{MainTheorem with P} is complete.
\qed

\medskip

\subsection{Proofs of Theorem \ref{First variant of MainTheorem with P}.} Up to replacing $m$ with $m^{*}$, using the projection $P$ whenever it is necessary, and some other trivial changes, the proofs of this result is the same as that of Theorem \ref{Corollary 2 to the main coercive result}. The details can be left to the reader.

\medskip

\subsection{Proof of Theorem \ref{MainTheorem with P for C1omegaloc}.} Up to replacing $|x|^2$ with $\theta(|x|)$, where $\theta(t):=\int_{0}^{t}\omega(s)ds$, and making some other obvious changes, the proof is the same as that of Theorem \ref{MainTheorem with P}. We leave it to the interested reader.

\medskip

\subsection{Proof of Theorem \ref{MainTheorem with P=I for C1}.} The proof of the sufficiency part follows the scheme of that of Theorem \ref{Corollary to C11 Whitney thm for coercive convex functions}, with some important changes which we next explain. 

We define the functions $m$ and $g$ as in the proof Theorem \ref{Corollary to C11 Whitney thm for coercive convex functions} (but recall that now $E$ is assumed to be closed and $f, G$ continuous). All the statements in that proof remain valid in our new setting until we arrive to \eqref{g=f on E}. At this point we need to replace Lemmas \ref{g is locally Lip and satisfies locally uniform supdiff estimates} and \ref{the convex envelope is smooth} with the following two lemmas.
\begin{lem}\label{gflat restricts in C1 proof}
For every $x\in X$ there exists some $\eta_x>0$ such that
$$
g(x)=\inf\{ f(y)+\langle G(y), x-y\rangle +\varphi_{y}(x)+a|x-y|^{2} \, : \, y\in E\cap B(0, \eta_x)\},
$$
and this infimum is attained.
\end{lem}
\begin{proof}
Let us write
$
\widetilde{\varphi}_y(x)=\varphi_y(x)+a|x-y|^2.
$
Take a point $y_0\in E$ and a number $\eta_x>0$ such that $\eta_x>|x|+
(\widetilde{\varphi}_{y_0}(x))^{1/2}$. Then, if $y\in E\setminus B(0, \eta_x)$,
\begin{eqnarray*}
& & f(y)+\langle G(y), x-y\rangle + \widetilde{\varphi}_y(x) \geq  f(y_0)+\langle G(y_0), x-y_0\rangle +a |x-y|^2 \\
& & \geq f(y_0)+\langle G(y_0), x-z_0\rangle +\widetilde{\varphi}_{y_0}(x).
\end{eqnarray*}
This shows that the infimum defining $g(x)$ restricts to the ball $B(0, \eta_x)$. Since the intersection of this ball with $E$ is compact and the functions involved are continuous, it is clear that the infimum is attained.
\end{proof}
\begin{lem}
For every $x\in\R^n$ there exists $\xi_x\in\R^n$ such that 
\begin{equation}\label{g is superdifferentiable in proof of thm C1}
\limsup_{h\to 0}\frac{g(x+h)-g(x)-\langle \xi_x, h\rangle}{|h|}\leq 0.
\end{equation}
In particular $g$ is continuous.
\end{lem}
\begin{proof}
We keep denoting $\widetilde{\varphi}_{y}(x)=\varphi_y(x)+a|x-y|^2$.
As noted in the preceding lemma, the infimum defining $g(x)$ is attained at, say, some point $y_x\in B(0, \eta_x)$. Let us put
$$
\xi_x := G(y_x)+\nabla\widetilde{\varphi}_{y_x}(x).
$$
We have
\begin{eqnarray*}
& & g(x+h)-g(x)-\langle \xi_x, h\rangle  \\
& & \leq f(y_x)+\langle G(y_x), x+h-y_x\rangle +\widetilde{\varphi}_{y_x}(x+h) \\
& & \,\,\, \,\,\,-f(y_x)-\langle G(y_x), x-y_x\rangle -\widetilde{\varphi}_{y_x}(x)-\langle G(y_x)+\nabla\widetilde{\varphi}_{y_x}(x), h\rangle \\
& &= \widetilde{\varphi}_{y_x}(x+h) -\widetilde{\varphi}_{y_x}(x)-\langle \nabla\widetilde{\varphi}_{y_x}(x), h\rangle \,=\, o(h).
\end{eqnarray*}
\end{proof}
We can then define
$
F=\textrm{conv}(g)
$
and use the remark made in \cite{KirchheimKristensen} that \eqref{g is superdifferentiable in proof of thm C1} together with
$
\lim_{|x|\to\infty}g(x)=\infty
$
are sufficient to ensure the differentiability of $F$. Since $F$ is convex, it follows that $F\in C^{1}(\R^n)$. The rest of the proof is exactly as in that of Theorem \ref{Corollary to C11 Whitney thm for coercive convex functions}.

The necessity part is obvious: just set $\varphi_y(x)=F(x)-F(y)-\langle \nabla F(y), x-y\rangle$.
\qed

\section{Acknowledgment}
I want to thank Pavel Shvartsman for reading this paper and making several suggestions that led me to improve the exposition. I also thank the referee for the same reason.





\end{document}